\documentclass[reqno,english]{amsart}
\usepackage{amsfonts,amsmath,latexsym,verbatim,amscd,mathrsfs,color,array}
\usepackage[colorlinks=true]{hyperref}
\usepackage{amssymb,amsthm,graphicx,color}
\usepackage[hmargin=3cm, vmargin=3cm]{geometry}
\usepackage{bbm}
\usepackage{amssymb}
\usepackage{pdfsync}
\usepackage{epstopdf}
\usepackage{cite}
\usepackage{graphicx}
\usepackage{multirow}
\usepackage[font=bf,aboveskip=15pt]{caption}
\usepackage[toc,page]{appendix}
\usepackage[colorlinks=true]{hyperref}
\usepackage{bookmark}
\DeclareFontShape{OT1}{cmr}{bx}{sc}{<-> cmbcsc10}{}
\usepackage{float}
\usepackage{ulem}
\usepackage{syntonly}
\usepackage{mathtools}
\usepackage{bm}
\usepackage{amsfonts,amsmath,latexsym,verbatim,amscd,mathrsfs,color,array}
\usepackage[colorlinks=true]{hyperref}

\usepackage{amsmath,amssymb,amsthm,amsfonts,graphicx,color}
\usepackage{amssymb}
\usepackage{pdfsync}
\usepackage{epstopdf}
\usepackage{upgreek}

\newcommand{\cuad}{{\sqcap\kern-.68em\sqcup}}

\newcommand{\be}{\begin{equation}}
\newcommand{\ee}{\end{equation}}

\newtheorem{Lemma}{Lemma}[section]
\newtheorem{Proposition}{Proposition}[section]
\newtheorem{Theorem}{Theorem}

\newtheorem{remark}{Remark}[section]
\newtheorem{exmp}{Example}[section]
\newcommand{\bremark}{\begin{remark} \em}
	\newcommand{\eremark}{\end{remark} }

\numberwithin{equation}{section}

\begin{document}
\title{ Sub-solutions and a point-wise Hopf's Lemma for Fractional p-Laplacian}

\author[Z. Li]{ Zaizheng Li}
\address{\noindent 
	Department of Mathematics, Hebei Normal University, Shijiazhuang 050024, China; Department of Mathematical Sciences, Yeshiva University}
\email{lizaizheng@amss.ac.cn}

\author[Q. Zhang]{Qidi Zhang}
\address{\noindent
	Department of Mathematics,
	University of British Columbia, Vancouver, B.C., V6T 1Z2, Canada}
\email{qidi@math.ubc.ca}
\iffalse\date{}
\author{\sffamily\Large Zaizheng Li$^{a,b}$\footnote{Corresponding author, lizaizheng@amss.ac.cn, supported by the CHINA SCHOLARSHIP COUNCIL}, Qidi Zhang$^{c}$\\
	{\sffamily $^a$ School of Mathematical Sciences, University of Chinese Academy of Sciences;}\\
	{\sffamily $^b$ Department of Mathematical Sciences, Yeshiva University. E-mail: lizaizheng@amss.ac.cn}\\
	{\sffamily $^c$  Department of Mathematics, University of Columbia British. E-mail: qidi@math.ubc.ca}}
\maketitle	
\fi
\begin{abstract}
	We prove a Hopf's lemma in the point-wise sense. The essential technique is to prove $(-\Delta)^s_p u(x)$ is uniformly bounded in the unit ball $B_1\subset\mathbb{R}^n$, where $u(x)=(1-|x|^2)^s_{+}$.  Also we study the global H\"older continuity of bounded positive solutions for $(-\Delta)^s_p u(x)=f(x,u).$
	
\end{abstract}
\maketitle

\tableofcontents	
\section{Introduction and  main results}
Fractional p-Laplacian	operator $(-\Delta)^s_p$ is a non-local operator, which is of the form
\begin{equation*}
\begin{aligned}
(-\Delta)^s_pu(x)&=C_{n,s,p}P.V.\int_{\mathbb{R}^n}\frac{|u(x)-u(y)|^{p-2}[u(x)-u(y)]}{|x-y|^{n+sp}}\mathrm{d}y\\
&=C_{n,s,p}\lim_{\epsilon\rightarrow0}\int_{\mathbb{R}^n\backslash B_\epsilon(x)}\frac{|u(x)-u(y)|^{p-2}[u(x)-u(y)]}{|x-y|^{n+sp}}\mathrm{d}y,
\end{aligned}
\end{equation*}
where $p\geq2,$ $s\in(0,1),$ $C_{n,s,p}$ is a constant and P.V. is the Cauchy's principal value.
The fractional p-Laplacian operator is an extension version of fractional Laplacian ($p=2$).	
In order that the integral on the right hand is well defined, we require  $$u\in C^{1,1}_{loc}\cap \mathcal{L}_{sp},$$ where $$\mathcal{L}_{sp}:=\left\{u\in L^{p-1}_{loc}\Big\vert\int_{\mathbb{R}^n} \frac{|1+u(x)|^{p-1}}{1+|x|^{n+sp}}<\infty\right\}.$$

In recent years, the non-local operators arise from many fields, such as game theory, finance, L\'evy processes, and optimization, see \cite{alberti1998, Sire2008, Bogdan_2010, bjorland2012, hitoshi2010, chen2020, Chen2018symmetry} and  references therein. In the special case $p=2,$  Luis Caffarelli and Luis Silvestre in \cite{caffarelli2007extension} introduce an extension method which turns the non-local operator into a local one in higher dimensions. Luis Silvestre  in\cite{silvestre2007regularity}, Ros-oton and Serra in\cite{ros2014dirichlet} discuss about the regularity of solutions for equations involving the fractional Laplacian. Wenxiong Chen, Congming Li and Yan Li in\cite{chen2017direct}, Wenxiong Chen, Congming Li and Shijie Qi in\cite{Chen2017direct2} develop  direct methods of moving planes and moving spheres. There are some explicit solutions for $p=2,$ for example, one can find $(-\Delta)^su(x)=const$ for $x\in B_1,$ where \begin{equation*}
u(x)=(1-|x|^2)^s_{+}=\begin{cases}
(1-|x|^2)^s,& |x|<1;\\
0,& |x|\geq1.
\end{cases}
\end{equation*} in  \cite{getoor1961first} and $(-\Delta)^s(x_n)^s_{+}=0,$ \begin{equation*}
(x_n)^s_{+}=\begin{cases}
x_n^s,& x_n>0;\\
0,& x_n\leq0.
\end{cases}
\end{equation*} in the upper half space  $\mathbb{R}^+_n$ in \cite{iannizzotto2014global}. We refer the reader to \cite{chen2020book, di2011hitchhiker} and references therein for more related results.

When $p>2$, $(-\Delta)^s_p$ is a non-local and nonlinear operator. In this case, $(-\Delta)^s_p (x_n)^s_{+}=0$ in the upper half space  $\mathbb{R}^+_n$ still holds, see \cite{iannizzotto2014global}. Our first result is the following sub-solutions result.
\begin{Theorem}[Sub-solutions]\label{thm 1}	
	Let $s\in(0,1),$ $p>2,$  $n\in\mathbb{N}^*,$ $u(x)=(1-|x|^2)^s_+,$ then $(-\Delta)^s_p u(x)$ is uniformly bounded in the unit ball $B_1\subset\mathbb{R}^n$.
\end{Theorem}
	This is really the first time we are able to prove the uniform boundedness of $(-\Delta)^s_p u(x),$ where $ u(x)=(1-|x|^2)^s_{+}$, and this property plays an essential role in the proof of Hopf's lemma. It is well known that the Hopf's lemma is one of the most useful tool in the theory of partial differential equations. For $p=2,$ the proof is based on Fourier transform and hypergeometric functions. But Fourier transform does not work anymore due to the nonlinearity when $p>2$, and we could not find out any hypergeometric function to exploit in this case. In fact, before this theorem, people even do not know whether $(-\Delta)^s_p u(x)$ is uniformly bounded. Compared with the case $p=2,$ $(-\Delta)^s_p u(x)$ is not constant anymore when $p>2$ by numerical calculation (see Proposition \ref{fulu}). Our proof is based on rigorous analysis on the singular term, then we figure out the exact coefficient of the singular term and we prove the coefficient is  $0$.
		
Our second result is the following point-wise Hopf's lemma. 
\begin{Theorem}[Hopf]\label{thm 2}
	Let $\Omega$ be a bounded domain with the uniform interior ball condition. For  $u(x)\in C^{1,1}_{loc}(\Omega)\cap \mathcal{L}_{sp}(\mathbb{R}^n),$ $p>2,$ $s\in(0,1)$. Assume $u$ is lower semicontinuous in $\overline{\Omega}$ and pointwisely satisfies
	\begin{equation}
	\begin{cases}
	(-\Delta)^s_p u+c(x)u\geq 0, & x\in \Omega,\\
	u>0, & x\in \Omega,\\
	u=0, & x\in \mathbb{R}^n\setminus\Omega,
	\end{cases}
	\end{equation}
where $c(x)\geq
	0$ is bounded. Then there exists a constant $C=C(\Omega, u)>0$, such that \[\liminf_{x\rightarrow \partial\Omega}\frac{u(x)}{\left[dist(x,\partial\Omega)\right]^s}\geq C.\] 	   	
\end{Theorem}
This is the first time to prove a point-wise Hopf's lemma for the fractional p-Laplacian. There are  some previous results about the Hopf's lemma. Wenxiong Chen and Congming Li in \cite{Chen2018Maximum} prove a boundary estimate for $(-\Delta)^s_p$, which is a key part in the moving plane method, and the boundary estimate plays the role of Hopf's lemma to some degree. Leandro M. Del pezzo and Alexander Quaas in \cite{del2017hopf}, Wenxiong Chen, Congming Li and Shijie Qi in \cite{chen2018hopf} prove a Hopf's lemma for $u\in\mathcal{W}^{s,p}(\Omega)$, where
$$\mathcal{W}^{s,p}(\Omega):=\{u\in L^p_{loc}(\mathbb{R}^n)\big\rvert \exists U\supset\supset \Omega, \quad such~that\\~ \iint\limits_{U\times U}\frac{|u(x)-u(y)|^p}{|x-y|^{n+sp}}\mathrm{d}x\mathrm{d}y<+\infty\}.$$  But $\mathcal{W}^{s,p}(\Omega)$ is different from $ C^{1,1}_{loc}(\Omega)\cap \mathcal{L}_{sp}(\mathbb{R}^n)$, see Example \ref{lizi}.
%Our proof is by constructing a sub-solution and then make use of comparison principle method.

The third result is the following H\"older regularity of positive solutions for $(-\Delta)^s_p u(x)=f(x,u)$ in any domain(bounded or not).
\begin{Theorem}\label{thm3}
	Let  $\Omega$ be any domain (bounded or not) with the uniform two-sided ball condition, $s\in(0,1)$, $p\geq2$, and  $u\in C^{1,1}_{loc}(\Omega)\cap \mathcal{L}_{sp}(\mathbb{R}^n)$ is a bounded positive solution of 
	\begin{equation}
	\begin{cases}
	(-\Delta)^s_p u=f(x,u), & x\in \Omega,\\
	%u>0, & x\in \Omega,\\
	u=0, & x\in \mathbb{R}^n\setminus\Omega,
	\end{cases}
	\end{equation}
	where $f(x,u)$ is bounded. Then there exists a constant $\nu_0\in(0,s),$ such that 
	$u\in C^{\nu_0}(\mathbb{R}^n)$. Moreover,\[\left[u\right]_{C^{\nu_0}(\mathbb{R}^n)}\leq C(\nu_0)\left[1+\|u\|_{L^\infty(\Omega)}+C\|f\|^{\frac{1}{p-1}}_{L^\infty(\Omega)}\right].\]
\end{Theorem}
\begin{remark}
%	When $u>0$ in $\Omega$, the boundary condition can be reduced to $u\leq 0$ in $\mathbb{R}^n\setminus\Omega.$
	See \cite{barb2009topics} for more information about the ball condition.
\end{remark}	
We extend the regularity results from bounded domains to unbounded domains. About the H\"older regularity of solutions for $(-\Delta)^s_p u(x)=f(x,u)$, Antonio Iannizzotto, Sunra Mosconi, and Marco Squassina in \cite{iannizzotto2014global} prove the global H\"older regularity of solutions in $\mathcal{W}^{s,p}_0(\Omega)$ in bounded domains. Lorenzo Brasco, Erik Lindgren, and Armin Schikorra in \cite{brasco2018higher} consider the higher H\"older regularity of local weak solutions in bounded domains, and they first give an explicit H\"older exponent. Yan Li and Lingyu Jin in \cite{jin2019} prove certain H\"older continuity up to the boundary in bounded domains. 

This paper is organized as follows. In section 2, we give some preliminary properties. Section 3 is devoted to showing Theorem \ref{thm 1}. Section 4 is contributed to proving Hopf's Theorem \ref{thm 2}. Section 5 is contributed to proving the global H\"older regularity for bounded positive solutions. In section 6, we list our numerical calculation results. The constant $C$ may vary from line to line or even in the same line. 

\section{Preliminaries}
Let's start by introducing some notations and properties we will use in this article.
\begin{Lemma}\label{duichen}
	Set $u(x)=(1-|x|^2)_+^s,$ then $(-\Delta)^s_p u(x)$ is radially symmetric in the unit ball.
\end{Lemma}
\begin{proof}
	Let $G[t]:=|t|^{p-2}t$, and $x\in B_1$ given. Set $A$ is any orthogonal transformation, then
	\begin{equation*}
	\begin{aligned}
	(-\Delta)^s_p u(x)&=C_{n,s,p}P.V.\int_{\mathbb{R}^n}\frac{G[(1-|x|^2)^s-(1-|y|^2)^s]}{|x-y|^{n+sp}}\mathrm{d}y\\
	&=C_{n,s,p}P.V.\int_{\mathbb{R}^n}\frac{G[(1-|Ax|^2)^s-(1-|Ay|^2)^s]}{|A(x-y)|^{n+sp}}\mathrm{d}y\\
	&=C_{n,s,p}P.V.\int_{\mathbb{R}^n}\frac{G[(1-|Ax|^2)^s-(1-|z|^2)^s]}{|Ax-z|^{n+sp}}\mathrm{d}z\\
	&=(-\Delta)^s_p u(Ax).
	\end{aligned}
	\end{equation*}
	which implies the radial symmetry of  $(-\Delta)^s_p u(x).$
\end{proof}
\begin{Lemma}\label{dandiao}
	Set $G(t):=[t]^{p-1}=|t|^{p-2}t$, $p\geq 2$, then $G(t)$ is strictly increasing and $$ G(t)-G(s)\leq 2^{2-p}G(t-s),\quad \forall t<s.$$  
\end{Lemma}
\begin{proof}
	\begin{itemize}
		\item[1)] For $t\neq0,$ by  $G^{\prime}(t)=(p-1)|t|^{p-2}>0,$ so $G(t)$ is strictly increasing.
		\item[2)] We only need to prove \[\frac{G(s)-G(t)}{G(s-t)}\geq 2^{2-p},\quad \forall t<s.\]
		If $s=0,$ it is obviously true. In the following we assume $s\neq0,$ then by direct calculation, $$\frac{G(s)-G(t)}{G(s-t)}=\frac{|s|^{p-2}s-|t|^{p-2}t}{|s-t|^{p-2}(s-t)}=\frac{1-|\frac{t}{s}|^{p-2}\frac{t}{s}}{|1-\frac{t}{s}|^{p-2}(1-\frac{t}{s})}.$$	
		Set \[F(\rho)=\frac{1-|\rho|^{p-2}\rho}{|1-\rho|^{p-2}(1-\rho)},\quad \rho\neq1.\]
		 It is readily to have 
 \[\lim\limits_{\rho\rightarrow \pm\infty}F(\rho)=1, \quad F'(\rho)=\frac{(p-1)(1-|\rho|^{p-2})}{|1-\rho|^p}.\]Hence for $p\geq 2$, $F(\rho)$ attains the minimum value at  $\rho=-1$, i.e. $F(\rho)\geq F(-1)=2^{2-p}.$ 
	\end{itemize} 
\end{proof}
\begin{Proposition}[Comparison principle]\label{bijiaoyuanli}
	Let $\Omega$ be a bounded domain, $u,v\in C^{1,1}_{loc}(\Omega)\cap \mathcal{L}_{sp}(\mathbb{R}^n)$, $u$ is lower semi-continuous in $\overline{\Omega}$ and $v$ is upper semi-continuous in $\overline{\Omega}$. $u, v$ satisfy
	\begin{equation}\label{bijiao}
	\begin{cases}
	(-\Delta)^s_p u(x)+c(x)u(x)\geq (-\Delta)^s_p v(x)+c(x)v(x),&\quad x\in\Omega,\\
	u(x)\geq v(x),& \quad x\in\mathbb{R}^n\setminus\Omega.
	\end{cases}
	\end{equation}
	where $c(x)\geq0$. Then $u(x)\geq v(x)$ in  $\Omega$.
\end{Proposition}
\begin{proof}
	We prove by contradiction.  If there is a point $x_0\in\Omega$, such that $u(x_0)<v(x_0)$. Since $u(x)$ is lower semi-continuous in $\overline{\Omega}$ and  $v(x)$ is upper semi-continuous in $\overline{\Omega}$, $u-v$ attains the minimum value in $\overline{\Omega}$. Without loss of generality, we assume \[0>u(x_0)-v(x_0)=\min\limits_{\Omega}(u-v)=\min\limits_{\mathbb{R}^n}(u-v).\] Then
	\[ u(x_0)-v(x_0)\leq u(y)-v(y),\quad \forall y\in\mathbb{R}^n.\]
	That is \[u(x_0)-u(y)\leq v(x_0)-v(y),\quad \forall y\in\mathbb{R}^n.\]
	By Lemma \ref{dandiao},
	\begin{equation}\label{bu}
	\begin{aligned}
	G\left(u(x_0)-u(y)\right)- G\left(v(x_0)-v(y)\right)&\leq0 ,\quad \forall y\in\mathbb{R}^n,\\
	G\left(u(x_0)-u(y)\right)- G\left(v(x_0)-v(y)\right)&<0 ,\quad \forall y\in\mathbb{R}^n\setminus\Omega.
	\end{aligned}
	\end{equation}
	By the equation \eqref{bijiao} and (\ref{bu}), we have 
	\begin{equation}
	\begin{aligned}
	0&\leq (-\Delta)^s_p u(x_0)- (-\Delta)^s_p v(x_0)+c(x_0)u(x_0)-c(x_0)v(x_0)\\
	&=C_{n,s,p}P.V.\int_{\mathbb{R}^n}\frac{G[u(x_0)-u(y)]}{|x_0-y|^{n+sp}}\mathrm{d}y-C_{n,s,p}P.V.\int_{\mathbb{R}^n}\frac{G[v(x_0)-v(y)]}{|x_0-y|^{n+sp}}\mathrm{d}y+c(x_0)\left(u(x_0)-v(x_0)\right)\\
	&=C_{n,s,p}P.V.\int_{\mathbb{R}^n}\frac{G[u(x_0)-u(y)]-G[v(x_0)-v(y)]}{|x_0-y|^{n+sp}}\mathrm{d}y+c(x_0)\left(u(x_0)-v(x_0)\right)\\
	&< 0.
	\end{aligned}
	\end{equation}
	Which is a contradiction. Thus there is no such $x_0$, and $u(x)\geq v(x)$ in  $\Omega$.
\end{proof}		
\section{Proof of Theorem \ref{thm 1}}	
In fact, for some  $\delta\in(0,1),$ when $|x|<1-\delta$, by \cite[Lemma 5.2]{Chen2018Maximum}, 
\begin{equation*}
\lvert(-\Delta)^s_p u(x)\rvert\leq C\int\limits_{\mathbb{R}^n\backslash B_\delta(x)}\frac{1}{|x-y|^{n+sp}}\mathrm{d}y+C\lvert\nabla u(x)\rvert^{p-2}\int_{B_\delta(x)}\frac{|x-y|^p}{|x-y|^{n+sp}}\mathrm{d}y\leq C.\\
\end{equation*}
That is,  $(-\Delta)^s_p u(x)$ is bounded for $|x|<1-\delta$. Therefore in the following we only need to consider the case when $|x|$ is close to 1.
%In this section, we will prove that $(-\Delta)^s_p u(x)$ is bounded in the ball $B_1$, where $u(x)=(1-x^2)^s_{+}$. Due to the symmetry, we only need to consider $x\in(0,1)$. Moreover, we only need to think about the case that $x$ is close to $1$. For some $\delta>0$ fixed, and $x<1-\delta$, we will treat this case in Section $3$ for all dimensions.
\subsection{$\mathbf{n=1}$}
In this part, for $u(x)=(1-x^2)^s_{+}$, we will prove that $(-\Delta)^s_p u(x)$ is uniformly bounded in $(-1,1)$. Due to Lemma \ref{duichen}, we only need to consider $x\in(0,1)$. The proof is divided into 3 steps.\\
Step 1.	
Firstly we give a general estimate for $(-\Delta)^s_p u(x)$ when $x$ is close to 1. For simplicity, we omit the constant $C_{n,s,p}.$
\begin{equation*}
\begin{aligned}
(-\Delta)^s_p u(x)&=\int_{-\infty}^{-1}\frac{(1-x^2)^{s(p-1)}}{(x-y)^{1+sp}}\mathrm{d}y+\int_{-1}^{-x}\frac{\left[(1-x^2)^s-(1-y^2)^s\right]^{p-1}}{(x-y)^{1+sp}}\mathrm{d}y+\int_{1}^{\infty}\frac{(1-x^2)^{s(p-1)}}{(y-x)^{1+sp}}\mathrm{d}y\\
&\quad+\lim_{\epsilon\rightarrow 0}\left\{\int_{-x}^{x-\epsilon}\frac{-\left[(1-y^2)^s-(1-x^2)^s\right]^{p-1}}{(x-y)^{1+sp}}\mathrm{d}y+\int_{x+\epsilon}^{1}\frac{\left[(1-x^2)^s-(1-y^2)^s\right]^{p-1}}{(y-x)^{1+sp}}\mathrm{d}y\right\}\\
&=\int_{1+x}^{\infty}\frac{(1-x^2)^{s(p-1)}}{z^{1+sp}}\mathrm{d}z+\int_{2x}^{1+x}\frac{\left[(1-x^2)^s-\left(1-(x-z)^2\right)^s\right]^{p-1}}{z^{1+sp}}\mathrm{d}z+\int_{1-x}^{\infty}\frac{(1-x^2)^{s(p-1)}}{z^{1+sp}}\mathrm{d}z\\
&\quad+\lim_{\epsilon\rightarrow 0}\left\{\int_{\epsilon}^{2x}\frac{-\left[\left(1-(x-z)^2\right)^s-(1-x^2)^s\right]^{p-1}}{z^{1+sp}}\mathrm{d}z+\int_{\epsilon}^{1-x}\frac{\left[(1-x^2)^s-\left(1-(x+z)^2\right)^s\right]^{p-1}}{z^{1+sp}}\mathrm{d}z\right\}\\
&=\int_{1+x}^{\infty}\frac{(1-x^2)^{s(p-1)}}{z^{1+sp}}\mathrm{d}z+\int_{2x}^{1+x}\frac{\left[(1-x^2)^s-\left(1-(x-z)^2\right)^s\right]^{p-1}}{z^{1+sp}}\mathrm{d}z\\
&\quad+\lim_{\epsilon\rightarrow 0}\int_{\epsilon}^{1-x}\frac{\left[(1-x^2)^s-\left(1-(x+z)^2\right)^s\right]^{p-1}-\left[\left(1-(x-z)^2\right)^s-(1-x^2)^s\right]^{p-1}}{z^{1+sp}}\mathrm{d}z\\
&\quad+\int_{1-x}^{1}\frac{(1-x^2)^{s(p-1)}-\left[(1-(x-z)^2)^s-(1-x^2)^s\right]^{p-1}}{z^{1+sp}}\mathrm{d}z\\
&\quad+\int_{1}^{2x}\frac{-\left[(1-(x-z)^2)^s-(1-x^2)^s\right]^{p-1}}{z^{1+sp}}\mathrm{d}z+\int_{1}^{\infty}\frac{(1-x^2)^{s(p-1)}}{z^{1+sp}}\mathrm{d}z\\
&=:I_1+I_2+I_3+I_4+I_5+I_6.
\end{aligned}
\end{equation*}
Where \begin{equation*}
\begin{aligned}
I_1+I_6&= (1-x^2)^{s(p-1)}\left[\int_{1+x}^{\infty}\frac{1}{z^{1+sp}}\mathrm{d}z+\int_{1}^{\infty}\frac{1}{z^{1+sp}}\mathrm{d}z\right]\\
&=(1-x^2)^{s(p-1)}\left[\frac{1}{sp(1+x)^{sp}}+\frac{1}{sp}\right],\\
\end{aligned}
\end{equation*}
and 
\begin{equation*}
\begin{aligned}
|I_2|+|I_5|&\leq \int_{2x}^{1+x}\frac{1}{z^{1+sp}}\mathrm{d}z+\int_{1}^{2x}\frac{1}{z^{1+sp}}\mathrm{d}z=\int_{1}^{1+x}\frac{1}{z^{1+sp}}\mathrm{d}z=\frac{1}{sp}\left[1-\frac{1}{(1+x)^{sp}}\right].
\end{aligned}
\end{equation*}
So $I_1,I_2,I_5,I_6$ are uniformly bounded. Then we only need to consider $I_3, I_4.$
\begin{equation*}
\begin{aligned}
I_3&=\lim_{\epsilon\rightarrow 0}\int_{\epsilon}^{1-x}\frac{\left[(1-x^2)^s-\left(1-(x+z)^2\right)^s\right]^{p-1}-\left[\left(1-(x-z)^2\right)^s-(1-x^2)^s\right]^{p-1}}{z^{1+sp}}\mathrm{d}z\\
&=(1-x^2)^{s(p-1)}\lim_{\epsilon\rightarrow 0}\int_{\epsilon}^{1-x}\frac{\left[1-\left(1-\frac{2xz}{1-x^2}-\frac{z^2}{1-x^2}\right)^s\right]^{p-1}-\left[\left(1+\frac{2xz}{1-x^2}-\frac{z^2}{1-x^2}\right)^s-1\right]^{p-1}}{z^{1+sp}}\mathrm{d}z\\
%&=\frac{(1+x)^{s(p-1)}}{(1-x)^s}\lim_{\epsilon\rightarrow 0}\int_{\frac{\epsilon}{1-x}}^{1}\frac{\left[1-\left(1-\frac{2xy}{1+x}-\frac{1-x}{1+x}y^2\right)^s\right]^{p-1}-\left[\left(1+\frac{2xy}{1+x}-\frac{1-x}{1+x}y^2\right)^s-1\right]^{p-1}}{y^{1+sp}}\mathrm{d}y\\
&=\frac{(2x)^{sp}}{(1-x^2)^s}\lim_{\epsilon\rightarrow 0}\int_{\frac{2x\epsilon}{1-x^2}}^{\frac{2x}{1+x}}\frac{\left[1-\left(1-k-\frac{1-x^2}{4x^2}k^2\right)^s\right]^{p-1}-\left[\left(1+k-\frac{1-x^2}{4x^2}k^2\right)^s-1\right]^{p-1}}{k^{1+sp}}\mathrm{d}k\\
&=:\frac{(2x)^{sp}}{(1-x)^s(1+x)^s}I_3',
\end{aligned}
\end{equation*}
where we use the substitution $$ k=\frac{2x}{1-x^2}z.$$
Similarly, we deal with $I_4,$
\begin{equation*}
\begin{aligned}
I_4&=\int_{1-x}^{1}\frac{(1-x^2)^{s(p-1)}-\left[\left(1-(x-z)^2\right)^s-(1-x^2)^s\right]^{p-1}}{z^{1+sp}}\mathrm{d}z\\
&=(1-x^2)^{s(p-1)}\left[\frac{1}{sp(1-x)^{sp}}-\frac{1}{sp}\right]-\int_{1-x}^{1}\frac{\left[\left(1-(x-z)^2\right)^s-(1-x^2)^s\right]^{p-1}}{z^{1+sp}}\mathrm{d}z\\
&=(1-x^2)^{s(p-1)}\left[\frac{1}{sp(1-x)^{sp}}-\frac{1}{sp}\right]-\frac{(2x)^{sp}}{(1-x^2)^s}\int_{\frac{2x}{1+x}}^{\frac{2x}{1-x^2}}\frac{\left[\left(1+k-\frac{1-x^2}{4x^2}k^2\right)^s-1\right]^{p-1}}{k^{1+sp}}\mathrm{d}k\\
&=\frac{1}{(1-x)^s}\left[\frac{(1+x)^{s(p-1)}}{sp}-\frac{(2x)^{sp}}{(1+x)^s}\int_{\frac{2x}{1+x}}^{\frac{2x}{1-x^2}}\frac{\left[(1+k-\frac{1-x^2}{4x^2}k^2)^s-1\right]^{p-1}}{k^{1+sp}}\mathrm{d}k\right]-\frac{(1-x^2)^{s(p-1)}}{sp}\\
&=:\frac{1}{(1-x)^s}\left[\frac{(1+x)^{s(p-1)}}{sp}-\frac{(2x)^{sp}}{(1+x)^s}I_4'\right]-\frac{(1-x^2)^{s(p-1)}}{sp}.\\
\end{aligned}
\end{equation*}

Step 2. The aim of this part is to simplify $(1-x)^{-s}I_3'$ and $(1-x)^{-s}I_4'.$ To be precise, we will prove 
\begin{equation}
\begin{aligned}
-C\leq(1-x)^{-s}I_3'-(1-x)^{-s}\int_{0}^{1}\frac{\left[1-\left(1-k\right)^s\right]^{p-1}-\left[\left(1+k\right)^s-1\right]^{p-1}}{k^{1+sp}}\mathrm{d}k\leq C,\\
\end{aligned}
\end{equation}
and 
\begin{equation}
\begin{aligned}
-C\leq(1-x)^{-s}I_4'-(1-x)^{-s}\int_{1}^{\infty}\frac{\left[(1+k)^s-1\right]^{p-1}}{k^{1+sp}}\mathrm{d}k\leq C.
\end{aligned}
\end{equation}
Firstly, we cope with the term $(1-x)^{-s}I_3',$
\begin{equation*}
\begin{aligned}
&\quad(1-x)^{-s}I_3'=(1-x)^{-s}\int_{0}^{\frac{2x}{1+x}}\frac{\left[1-\left(1-k-\frac{1-x^2}{4x^2}k^2\right)^s\right]^{p-1}-\left[\left(1+k-\frac{1-x^2}{4x^2}k^2\right)^s-1\right]^{p-1}}{k^{1+sp}}\mathrm{d}k\\
&\leq(1-x)^{-s}\int_{0}^{1}\frac{\left[1-\left(1-k-\frac{1-x^2}{4x^2}k^2\right)^s\right]^{p-1}-\left[\left(1+k-\frac{1-x^2}{4x^2}k^2\right)^s-1\right]^{p-1}}{k^{1+sp}}\mathrm{d}k+C.\\
\end{aligned}
\end{equation*}
This is because 
\begin{equation*}
\begin{aligned}
&\quad\limsup\limits_{x\rightarrow 1}(1-x)^{-s}\Bigg\lvert\int_{\frac{2x}{1+x}}^{1}\frac{\left[1-\left(1-k-\frac{1-x^2}{4x^2}k^2\right)^s\right]^{p-1}-\left[\left(1+k-\frac{1-x^2}{4x^2}k^2\right)^s-1\right]^{p-1}}{k^{1+sp}}\mathrm{d}k\Bigg\rvert\\
&\leq \limsup\limits_{x\rightarrow 1}(1-x)^{-s}\frac{(1+x)^{sp}}{(2x)^{sp}}\frac{1-x}{1+x}=0.
\end{aligned}
\end{equation*}
Next we are going to prove there exists a constant $C>0,$ such that
\begin{equation*}
\begin{aligned}
-C\leq&(1-x)^{-s}\int_{0}^{1}\frac{\left[1-\left(1-k-\frac{1-x^2}{4x^2}k^2\right)^s\right]^{p-1}-\left[\left(1+k-\frac{1-x^2}{4x^2}k^2\right)^s-1\right]^{p-1}}{k^{1+sp}}\mathrm{d}k\\
&-(1-x)^{-s}\int_{0}^{1}\frac{\left[1-\left(1-k\right)^s\right]^{p-1}-\left[\left(1+k\right)^s-1\right]^{p-1}}{k^{1+sp}}\mathrm{d}k\leq C.\\
\end{aligned}
\end{equation*}
Just consider the difference above, 
\begin{equation*}
\begin{aligned}
&\quad(1-x)^{-s}\int_{0}^{1}\frac{\left[1-\left(1-k-\frac{1-x^2}{4x^2}k^2\right)^s\right]^{p-1}-\left[1-\left(1-k\right)^s\right]^{p-1}}{k^{1+sp}}\mathrm{d}k\\
&=(1-x)^{-s}\int_{0}^{\frac{1}{2}}\frac{\left[1-\left(1-k-\frac{1-x^2}{4x^2}k^2\right)^s\right]^{p-1}-\left[1-\left(1-k\right)^s\right]^{p-1}}{k^{1+sp}}\mathrm{d}k\\
&\quad+(1-x)^{-s}\int_{\frac{1}{2}}^{1}\frac{\left[1-\left(1-k-\frac{1-x^2}{4x^2}k^2\right)^s\right]^{p-1}-\left[1-\left(1-k\right)^s\right]^{p-1}}{k^{1+sp}}\mathrm{d}k\\
&\leq \int_{0}^{\frac{1}{2}}\frac{C k^{s(p-2)}\left[\left(1-k\right)^s-\left(1-k-\frac{1-x^2}{4x^2}k^2\right)^s\right]}{k^{1+sp}}\mathrm{d}k\\
&\quad+(1-x)^{-s}\int_{\frac{1}{2}}^{1}\frac{C k^{s(p-2)}\left[\left(1-k\right)^s-\left(1-k-\frac{1-x^2}{4x^2}k^2\right)^s\right]}{k^{1+sp}}\mathrm{d}k\\
&\leq (1-x)^{-s}\int_{0}^{\frac{1}{2}}\frac{C k^{s(p-2)}\left(1-k\right)^s\left[1-\left(1-\frac{1-x^2}{4x^2}\frac{k^2}{1-k}\right)^s\right]}{k^{1+sp}}\mathrm{d}k+(1-x)^{-s}\int_{\frac{1}{2}}^{1}\frac{C k^{s(p-2)}\frac{(1-x^2)^sk^{2s}}{(4x^2)^s}}{k^{1+sp}}\mathrm{d}k\\
&\leq (1-x)^{-s}\int_{0}^{\frac{1}{2}}\frac{C k^{s(p-2)}\left(1-k\right)^s\left[s\frac{1-x^2}{4x^2}\frac{k^2}{1-k}+o\left(\frac{1-x^2}{4x^2}\frac{k^2}{1-k}\right)\right]}{k^{1+sp}}\mathrm{d}k+C\leq C.\\
\end{aligned}
\end{equation*}
And similarly,
\begin{equation*}
\begin{aligned}
&\quad(1-x)^{-s}\int_{0}^{1}\frac{\left[1-\left(1+k-\frac{1-x^2}{4x^2}k^2\right)^s\right]^{p-1}-\left[1-\left(1+k\right)^s\right]^{p-1}}{k^{1+sp}}\mathrm{d}k\\
%&=(1-x)^{-s}\int_{0}^{\frac{1}{2}}\frac{\left[1-\left(1+k-\frac{1-x^2}{4x^2}k^2\right)^s\right]^{p-1}-\left[1-\left(1+k\right)^s\right]^{p-1}}{k^{1+sp}}\mathrm{d}k\\
%&\quad+(1-x)^{-s}\int_{\frac{1}{2}}^{1}\frac{\left[1-\left(1+k-\frac{1-x^2}{4x^2}k^2\right)^s\right]^{p-1}-\left[1-\left(1+k\right)^s\right]^{p-1}}{k^{1+sp}}\mathrm{d}k\\
%&\leq (1-x)^{-s}\int_{0}^{\frac{1}{2}}\frac{C k^{s(p-2)}\left[\left(1+k\right)^s-\left(1+k-\frac{1-x^2}{4x^2}k^2\right)^s\right]}{k^{1+sp}}\mathrm{d}k\\
%&\quad+(1-x)^{-s}\int_{\frac{1}{2}}^{1}\frac{C k^{s(p-2)}\left[\left(1+k\right)^s-\left(1+k-\frac{1-x^2}{4x^2}k^2\right)^s\right]}{k^{1+sp}}\mathrm{d}k\\
%&\leq (1-x)^{-s}\int_{0}^{\frac{1}{2}}\frac{C k^{s(p-2)}\left(1+k\right)^s\left[1-\left(1-\frac{1-x^2}{4x^2}\frac{k^2}{1+k}\right)^s\right]}{k^{1+sp}}\mathrm{d}k\\
%&\quad+(1-x)^{-s}\int_{\ frac{1}{2}}^{1}\frac{C k^{s(p-2)}\frac{(1-x^2)^sk^{2s}}{(4x^2)^s}}{k^{1+sp}}\mathrm{d}k\\
&\leq (1-x)^{-s}\int_{0}^{\frac{1}{2}}\frac{C k^{s(p-2)}\left(1-k\right)^s\left[s\frac{1-x^2}{4x^2}\frac{k^2}{1+k}+o\left(\frac{1-x^2}{4x^2}\frac{k^2}{1+k}\right)\right]}{k^{1+sp}}\mathrm{d}k+C\leq C.\\
\end{aligned}
\end{equation*}
Secondly, we estimate the term $(1-x)^{-s}I_4',$
\begin{equation*}
\begin{aligned}
I_4'=\int_{\frac{2x}{1+x}}^{\frac{2x}{1-x^2}}\frac{\left[(1+k-\frac{1-x^2}{4x^2}k^2)^s-1\right]^{p-1}}{k^{1+sp}}\mathrm{d}k.
\end{aligned}
\end{equation*}
Next we will prove 
\begin{equation*}
\begin{aligned}
&\quad(1-x)^{-s}I_4'=(1-x)^{-s}\int_{\frac{2x}{1+x}}^{\frac{2x}{1-x^2}}\frac{\left[(1+k-\frac{1-x^2}{4x^2}k^2)^s-1\right]^{p-1}}{k^{1+sp}}\mathrm{d}k\\
&\leq(1-x)^{-s}\int_{1}^{\frac{1}{1-x}}\frac{\left[(1+k-\frac{1-x^2}{4x^2}k^2)^s-1\right]^{p-1}}{k^{1+sp}}\mathrm{d}k+C.\\
\end{aligned}
\end{equation*}
This is due to 
\begin{equation*}
\begin{aligned}
&\quad\limsup\limits_{x\rightarrow 1}(1-x)^{-s}\int_{\frac{2x}{1+x}}^{1}\frac{\left[(1+k-\frac{1-x^2}{4x^2}k^2)^s-1\right]^{p-1}}{k^{1+sp}}\mathrm{d}k\\
&\leq\limsup\limits_{x\rightarrow 1} C(1-x)^{-s}\frac{1-x}{1+x}=0,\\
\end{aligned}
\end{equation*}
and 
\begin{equation*}
\begin{aligned}
&\quad\limsup\limits_{x\rightarrow 1}(1-x)^{-s}\Bigg\lvert\int_{\frac{2x}{1-x^2}}^{\frac{1}{1-x}}\frac{\left[(1+k-\frac{1-x^2}{4x^2}k^2)^s-1\right]^{p-1}}{k^{1+sp}}\mathrm{d}k\Bigg\rvert\\
&\leq\limsup\limits_{x\rightarrow 1} C(1-x)^{-s}\int_{\frac{2x}{1-x^2}}^{\frac{1}{1-x}}\frac{k^{s(p-1)}}{k^{1+sp}}\mathrm{d}k\\
&=\limsup\limits_{x\rightarrow 1} C(1-x)^{-s}\frac{(1-x^2)^{1+s}}{(2x)^{1+s}}\left(\frac{1}{1-x}-\frac{2x}{1-x^2}\right)=0.
\end{aligned}
\end{equation*}
Now we claim that there exists a constant $C>0,$ such that 
\begin{equation*}
\begin{aligned}
-C\leq&(1-x)^{-s}\int_{1}^{\frac{1}{1-x}}\frac{\left[(1+k-\frac{1-x^2}{4x^2}k^2)^s-1\right]^{p-1}}{k^{1+sp}}\mathrm{d}k-(1-x)^{-s}\int_{1}^{\frac{1}{1-x}}\frac{\left[(1+k)^s-1\right]^{p-1}}{k^{1+sp}}\mathrm{d}k\leq C.\\
\end{aligned}
\end{equation*}
Because $x$ is close to $1$, we have
\begin{equation*}
\begin{aligned}
&\quad(1-x)^{-s}\int_{1}^{\frac{1}{1-x}}\frac{\left[\left(1+k\right)^s-1\right]^{p-1}-\left[\left(1+k-\frac{1-x^2}{4x^2}k^2\right)^s-1\right]^{p-1}}{k^{1+sp}}\mathrm{d}k\\
%&=(1-x)^{-s}\int_{1}^{\frac{1}{1-x}}\frac{\left[\left(1+k\right)^s-1\right]^{p-1}-\left[\left(1+k-\frac{1-x^2}{4x^2}k^2\right)^s-1\right]^{p-1}}{k^{1+sp}}\mathrm{d}k\\
&\leq (1-x)^{-s}\int_{1}^{\frac{1}{1-x}}\frac{Ck^{s(p-2)}\left[\left(1+k\right)^s-\left(1+k-\frac{1-x^2}{4x^2}k^2\right)^s\right]}{k^{1+sp}}\mathrm{d}k\\
&\leq (1-x)^{-s}\int_{1}^{\frac{1}{1-x}}\frac{Ck^{s(p-2)}\left(1+k\right)^s\left[1-\left(1-\frac{1-x^2}{4x^2}\frac{k^2}{1+k}\right)^s\right]}{k^{1+sp}}\mathrm{d}k\\
&\leq (1-x)^{-s}\int_{1}^{\frac{1}{1-x}}\frac{Ck^{s(p-2)}\left(1+k\right)^s\left[s\frac{1-x^2}{4x^2}\frac{k^2}{1+k}+o\left(\frac{1-x^2}{4x^2}\frac{k^2}{1+k}\right)\right]}{k^{1+sp}}\mathrm{d}k\\
&\leq (1-x)^{-s}\int_{1}^{\frac{1}{1-x}}\frac{Ck^{s(p-2)}(1-x)k^{1+s}}{k^{1+sp}}\mathrm{d}k\\
&=(1-x)^{1-s}\int_{1}^{\frac{1}{1-x}}Ck^{-s}\mathrm{d}k
\leq C(1-x)^{1-s}\frac{1}{(1-x)^{1-s}}=C,
\end{aligned}
\end{equation*}
where we use the estimate $$\frac{1-x^2}{4x^2}\frac{k^2}{1+k}\leq\frac{1+x}{4x^2(2-x)}\rightarrow \frac{1}{2} \quad\text{if} \quad  x\rightarrow 1.$$
And there exists a positive constant $C$, such that  
\begin{equation*}
\begin{aligned}
-C\leq(1-x)^{-s}\int_{1}^{\frac{1}{1-x}}\frac{\left[(1+k)^s-1\right]^{p-1}}{k^{1+sp}}\mathrm{d}k-(1-x)^{-s}\int_{1}^{\infty}\frac{\left[(1+k)^s-1\right]^{p-1}}{k^{1+sp}}\mathrm{d}k\leq C.\\
\end{aligned}
\end{equation*}
This is because when $x$ is close to $1$, we have
\begin{equation*}
\begin{aligned}
&(1-x)^{-s}\int_{\frac{1}{1-x}}^{\infty}\frac{\left[(1+k)^s-1\right]^{p-1}}{k^{1+sp}}\mathrm{d}k\\
&\leq (1-x)^{-s}\int_{\frac{1}{1-x}}^{\infty}\frac{k^{s(p-1)}}{k^{1+sp}}\mathrm{d}k\\
&=(1-x)^{-s}\int_{\frac{1}{1-x}}^{\infty}k^{-s-1}\mathrm{d}k\leq C.\\
\end{aligned}
\end{equation*}

Step 3. We will prove all terms are bounded uniformly. By the above simplification, the singular term of $(-\Delta)^s_p u(x)$ is 
$$(1-x)^{-s}\frac{(2x)^{sp}}{(1+x)^s}\left\{\int_{0}^{1}\frac{\left[1-(1-k)^s\right]^{p-1}-\left[(1+k)^s-1\right]^{p-1}}{k^{1+sp}}\mathrm{d}k+\frac{(1+x)^{sp}}{sp(2x)^{sp}}-\int_{1}^{\infty}\frac{\left[(1+k)^s-1\right]^{p-1}}{k^{1+sp}}\mathrm{d}k\right\}.$$ 
Furthermore, there is a constant $C>0,$
\begin{equation*}
\begin{aligned}
-C\leq (1-x)^{-s}\frac{(1+x)^{sp}}{sp(2x)^{sp}}-(1-x)^{-s}\frac{1}{sp}\leq C.
\end{aligned}
\end{equation*}
Because 
\begin{equation*}
\begin{aligned}
(1-x)^{-s}\frac{1}{sp}\left[\frac{(1+x)^{sp}}{(2x)^{sp}}-1\right]&=(1-x)^{-s}\frac{1}{sp}\left[\left(1+\frac{1-x}{2x}\right)^{sp}-1\right]\\
&=(1-x)^{-s}\frac{1}{sp}\left[sp\frac{1-x}{2x}+o(\frac{1-x}{2x})\right]\leq C.
\end{aligned}
\end{equation*}
Hence we just need to prove the following identity:
\begin{equation}\label{identity}
\begin{aligned}
\frac{1}{sp}+\int_{0}^{1}\frac{\left[1-(1-k)^s\right]^{p-1}-\left[(1+k)^s-1\right]^{p-1}}{k^{1+sp}}\mathrm{d}k-\int_{1}^{\infty}\frac{\left[(1+k)^s-1\right]^{p-1}}{k^{1+sp}}\mathrm{d}k=0.
\end{aligned}
\end{equation}
Now we begin to prove the above identity. For any $\epsilon\in(0,1)$ fixed, we have 
\begin{equation*}
\begin{aligned}
&\quad\int_{0}^{1}\frac{\left[1-(1-k)^s\right]^{p-1}-\left[(1+k)^s-1\right]^{p-1}}{k^{1+sp}}\mathrm{d}k-\int_{1}^{\infty}\frac{\left[(1+k)^s-1\right]^{p-1}}{k^{1+sp}}\mathrm{d}k\\
&=\int_{0}^{\epsilon}\frac{\left[1-(1-k)^s\right]^{p-1}-\left[(1+k)^s-1\right]^{p-1}}{k^{1+sp}}\mathrm{d}k+\int_{\epsilon}^{1}\frac{\left[1-(1-k)^s\right]^{p-1}-\left[(1+k)^s-1\right]^{p-1}}{k^{1+sp}}\mathrm{d}k\\
&\quad-\int_{1}^{\infty}\frac{\left[(1+k)^s-1\right]^{p-1}}{k^{1+sp}}\mathrm{d}k\\
&=\int_{0}^{\epsilon}\frac{\left[1-(1-k)^s\right]^{p-1}-\left[(1+k)^s-1\right]^{p-1}}{k^{1+sp}}\mathrm{d}k+\int_{\epsilon}^{1}\frac{\left[1-(1-k)^s\right]^{p-1}}{k^{1+sp}}\mathrm{d}k-\int_{\epsilon}^{\infty}\frac{\left[(1+k)^s-1\right]^{p-1}}{k^{1+sp}}\mathrm{d}k.
\end{aligned}
\end{equation*}
Notice that
\begin{equation*}
\begin{aligned}
\int_{\epsilon}^{1}\frac{\left[1-(1-k)^s\right]^{p-1}}{k^{1+sp}}\mathrm{d}k=\int_{\frac{\epsilon}{1-\epsilon}}^{\infty}\frac{(1+t)^{s-1}\left[(1+t)^s-1\right]^{p-1}}{t^{1+sp}}\mathrm{d}t,
\end{aligned}
\end{equation*}
where we use the change of variable $$t=\frac{k}{1-k}.$$
Then 
\begin{equation*}
\begin{aligned}
&\quad\int_{0}^{1}\frac{\left[1-(1-k)^s\right]^{p-1}-\left[(1+k)^s-1\right]^{p-1}}{k^{1+sp}}\mathrm{d}k-\int_{1}^{\infty}\frac{\left[(1+k)^s-1\right]^{p-1}}{k^{1+sp}}\mathrm{d}k\\
&=\int_{0}^{\epsilon}\frac{\left[1-(1-k)^s\right]^{p-1}-\left[(1+k)^s-1\right]^{p-1}}{k^{1+sp}}\mathrm{d}k+\int_{\frac{\epsilon}{1-\epsilon}}^{\infty}\frac{(1+t)^{s-1}\left[(1+t)^s-1\right]^{p-1}}{t^{1+sp}}\mathrm{d}t\\
&\quad-\int_{\frac{\epsilon}{1-\epsilon}}^{\infty}\frac{\left[(1+t)^s-1\right]^{p-1}}{t^{1+sp}}\mathrm{d}t-\int_{\epsilon}^{\frac{\epsilon}{1-\epsilon}}\frac{\left[(1+k)^s-1\right]^{p-1}}{k^{1+sp}}\mathrm{d}k\\
&=\int_{0}^{\epsilon}\frac{\left[1-(1-k)^s\right]^{p-1}-\left[(1+k)^s-1\right]^{p-1}}{k^{1+sp}}\mathrm{d}k-\int_{\epsilon}^{\frac{\epsilon}{1-\epsilon}}\frac{\left[(1+k)^s-1\right]^{p-1}}{k^{1+sp}}\mathrm{d}k\\
&\quad-\int_{\frac{\epsilon}{1-\epsilon}}^{\infty}\frac{\left[(1+t)^s-1\right]^{p-1}\left[1-(1+t)^{s-1}\right]}{t^{1+sp}}\mathrm{d}t\\
&=\int_{0}^{\epsilon}\frac{\left[1-(1-k)^s\right]^{p-1}-\left[(1+k)^s-1\right]^{p-1}}{k^{1+sp}}\mathrm{d}k-\int_{\epsilon}^{\frac{\epsilon}{1-\epsilon}}\frac{\left[(1+k)^s-1\right]^{p-1}}{k^{1+sp}}\mathrm{d}k-\frac{1}{sp}\frac{\left[(1+t)^s-1\right]^{p}}{ t^{sp}}\Bigg\rvert_{\frac{\epsilon}{1-\epsilon}}^{\infty}\\
&=:H_1+H_2+H_3.
\end{aligned}
\end{equation*}
Now we estimate term by term and let $\epsilon\rightarrow0.$
\begin{equation*}
\begin{aligned}
H_1&=\int_{0}^{\epsilon}\frac{\left[1-(1-k)^s\right]^{p-1}-\left[(1+k)^s-1\right]^{p-1}}{k^{1+sp}}\mathrm{d}k\\
&\leq \int_{0}^{\epsilon}\frac{Ck^{s(p-2)}\left[2-(1-k)^s-(1+k)^s\right]}{k^{1+sp}}\mathrm{d}k\\
&\leq \int_{0}^{\epsilon}\frac{Ck^{s(p-2)}\left[2-\left(1-sk-\frac{s(1-s)}{2}k^2+o(k^2)\right)-\left(1+sk-\frac{s(1-s)}{2}k^2+o(k^2)\right)\right]}{k^{1+sp}}\mathrm{d}k\\
&\leq \int_{0}^{\epsilon}\frac{Ck^{s(p-2)}\left[k^2+o(k^2)\right]}{k^{1+sp}}\mathrm{d}k\rightarrow 0 \quad \text{as} \quad \epsilon\rightarrow 0.\\
%\end{aligned}
%\end{equation*}
%
%\begin{equation*}
%\begin{aligned}
H_2&=\int_{\epsilon}^{\frac{\epsilon}{1-\epsilon}}\frac{\left[(1+k)^s-1\right]^{p-1}}{k^{1+sp}}\mathrm{d}k=\int_{\epsilon}^{\frac{\epsilon}{1-\epsilon}}\frac{\left[\left(1+sk+o(k)\right)-1\right]^{p-1}}{k^{1+sp}}\mathrm{d}k\leq \int_{\epsilon}^{\frac{\epsilon}{1-\epsilon}}\frac{Ck^{p-1}}{k^{1+sp}}\mathrm{d}k\\
&=\begin{cases}
\int_{\epsilon}^{\frac{\epsilon}{1-\epsilon}}Ck^{p-sp-2}\mathrm{d}k\leq C\left(\frac{\epsilon}{1-\epsilon}\right)^{p-sp-1}\rightarrow 0 \quad \text{as} \quad \epsilon\rightarrow 0, & \text{if}\quad p-sp>1;\\
\int_{\epsilon}^{\frac{\epsilon}{1-\epsilon}}C\frac{1}{k}\mathrm{d}k\leq C\frac{1}{\epsilon}(\frac{\epsilon}{1-\epsilon}-\epsilon)=C\frac{\epsilon}{1-\epsilon}\rightarrow 0 \quad \text{as} \quad \epsilon\rightarrow 0,& \textrm{if}\quad p-sp=1;\\
\int_{\epsilon}^{\frac{\epsilon}{1-\epsilon}}C\frac{1}{k^{2-(p-sp)}}\mathrm{d}k\leq
C\frac{1}{\epsilon^{2-(p-sp)}}(\frac{\epsilon}{1-\epsilon}-\epsilon)=C\frac{\epsilon^{p-sp}}{1-\epsilon}\rightarrow 0 \quad \text{as} \quad \epsilon\rightarrow 0,& \text{if}\quad p-sp<1.
\end{cases}\\
H_3&=-\frac{1}{sp}\frac{\left[(1+t)^s-1\right]^{p}}{ t^{sp}}\Bigg\rvert_{\frac{\epsilon}{1-\epsilon}}^{\infty}=-\frac{1}{sp}+\frac{1}{sp}\frac{\left[1-(1-\epsilon)^s\right]^p}{\epsilon^{sp}}\\
&=-\frac{1}{sp}+\frac{1}{sp}\frac{\left[1-(1-s\epsilon-o(\epsilon))\right]^p}{\epsilon^{sp}}\rightarrow-\frac{1}{sp}  \quad \text{as} \quad \epsilon\rightarrow 0.\\
\end{aligned}
\end{equation*}
Therefore we have proved the identity (\ref{identity}).  Hence we have completed the proof for $n=1$.
\subsection{$\mathbf{n \geq 2}$}
The purpose of this section is to prove that for $u(x)=(1-|x|^2)^s_+$, $(-\Delta)^s_pu(x)$  is uniformly bounded in $B_1$ for higher dimensions. There are totally 14 steps.	

Step 1.  Due to Lemma \ref{duichen}, we assume $x:=(x,0,\dots,0)\in B_1$ is close to $(1,0,\dots,0),$ and  $ y=(y_1,y_2,\dots,y_n)=:(y_1,\bar{y}).$ We omit the constant $C_{n,s,p}$ for simplicity.
\begin{equation*}
\begin{aligned}
&(-\Delta)^s_pu(x)\\
= \ &\lim_{\epsilon\rightarrow0}\int\limits_{\mathbb{R}^n\backslash B_\epsilon(x)}\frac{|u(x)-u(y)|^{p-2}[u(x)-u(y)]}{|x-y|^{n+sp}}\mathrm{d}y\\
= \ &\lim_{\epsilon\rightarrow0}\int\limits_{\{y\in\mathbb{R}^n: |x-y|\geq \epsilon\}}\frac{\left|(1-x^2)^s-(1-|y|^2)^s_+\right|^{p-2}\left[(1-x^2)^s-(1-|y|^2)^s_+\right]}{|x-y|^{n+sp}}\mathrm{d}y\\
= \ &\lim_{\epsilon\rightarrow0}\int\limits_{\{y\in\mathbb{R}^n: |x-y|\geq \epsilon\}}\frac{\left|(1-x^2)^s-(1-y_1^2-|\bar{y}|^2)^s_+\right|^{p-2}\left[(1-x^2)^s-(1-y_1^2-|\bar{y}|^2)^s_+\right]}{\left|(x-y_1)^2+|\bar{y}|^2\right|^{\frac{n+sp}{2}}}\mathrm{d}y.
\end{aligned}
\end{equation*}
Set $\displaystyle z=(z_1,\bar{z}),$ where $\displaystyle z_1=x-y_1, \bar{z}=\bar{y}.$ Then
\begin{equation*}
\begin{aligned}
&(-\Delta)^s_pu(x)\\
= \ &\lim_{\epsilon\rightarrow0}\int\limits_{\{z\in\mathbb{R}^n: |z|\geq \epsilon\}}\frac{\left|(1-x^2)^s-\left(1-(x-z_1)^2-|\bar{z}|^2\right)^s_+\right|^{p-2}\left[(1-x^2)^s-\left(1-(x-z_1)^2-|\bar{z}|^2\right)^s_+\right]}{\left(z_1^2+|\bar{z}|^2\right)^{\frac{n+sp}{2}}}\mathrm{d}z\\
= \ &(1-x^2)^{s(p-1)}\lim_{\epsilon\rightarrow0}\int\limits_{\{z\in\mathbb{R}^n: |z|\geq \epsilon\}}\frac{\left|1-\left(1+\frac{2xz_1}{1-x^2}-\frac{|z|^2}{1-x^2}|\right)^s_+\right|^{p-2}\left[1-\left(1+\frac{2xz_1}{1-x^2}-\frac{|z|^2}{1-x^2}\right)^s_+\right]}{|z|^{n+sp}}\mathrm{d}z.\\
\end{aligned}
\end{equation*}
Let $\displaystyle w=\frac{2xz}{1-x^2},$ where $\displaystyle w_1=\frac{2xz_1}{1-x^2}$ and $\displaystyle \bar{w}=\frac{2x\bar{z}}{1-x^2}$, then 
\begin{equation}\label{zuichuguji}
\begin{aligned}
(-\Delta)^s_pu(x)
&=\frac{(2x)^{sp}}{(1-x^2)^s}\lim_{\epsilon\rightarrow0}\int\limits_{\{w\in\mathbb{R}^n: |w|\geq\frac{2x\epsilon}{1-x^2} \}}\frac{\left|1-\left(1+w_1-\frac{1-x^2}{4x^2}|w|^2\right)^s_+\right|^{p-2}\left[1-\left(1+w_1-\frac{1-x^2}{4x^2}|w|^2\right)^s_+\right]}{|w|^{n+sp}}\mathrm{d}w.
\end{aligned}
\end{equation}
Now we roughly  analyze the integration term.
\begin{equation*}
\begin{aligned}
&\quad\int\limits_{ |w|\geq\frac{2x\epsilon}{1-x^2} }\frac{\left|1-\left(1+w_1-\frac{1-x^2}{4x^2}|w|^2\right)^s_+\right|^{p-2}\left[1-\left(1+w_1-\frac{1-x^2}{4x^2}|w|^2\right)^s_+\right]}{|w|^{n+sp}}\mathrm{d}w\\
= \ &\int_{-\infty}^{\infty}\mathrm{d}w_1\int\limits_{ |\bar{w}|^2\geq\frac{4x^2\epsilon^2}{(1-x^2)^2}-w_1^2}\frac{\left|1-\left(1+w_1-\frac{1-x^2}{4x^2}w_1^2-\frac{1-x^2}{4x^2}|\bar{w}|^2\right)^s_+\right|^{p-2}\left[1-\left(1+w_1-\frac{1-x^2}{4x^2}w_1^2-\frac{1-x^2}{4x^2}|\bar{w}|^2\right)^s_+\right]}{\left(w_1^2+|\bar{w}|^2\right)^{\frac{n+sp}{2}}}\mathrm{d}\bar{w}\\
= \ &C\int_{-\infty}^{\infty}\mathrm{d}w_1\int\limits_{ \rho^2\geq\frac{4x^2\epsilon^2}{(1-x^2)^2}-w_1^2}\frac{\left|1-\left(1+w_1-\frac{1-x^2}{4x^2}w_1^2-\frac{1-x^2}{4x^2}\rho^2\right)^s_+\right|^{p-2}\left[1-\left(1+w_1-\frac{1-x^2}{4x^2}w_1^2-\frac{1-x^2}{4x^2}\rho^2\right)^s_+\right]}{\left(w_1^2+\rho^2\right)^{\frac{n+sp}{2}}}\rho^{n-2}\mathrm{d}\rho\\
= \ &C\int_{-\frac{2x\epsilon}{1-x^2}}^{\frac{2x\epsilon}{1-x^2}}\mathrm{d}w_1\int_{\sqrt{\frac{4x^2\epsilon^2}{(1-x^2)^2}-w_1^2} }^{\infty}\frac{\left|1-\left(1+w_1-\frac{1-x^2}{4x^2}w_1^2-\frac{1-x^2}{4x^2}\rho^2\right)^s_+\right|^{p-2}\left[1-\left(1+w_1-\frac{1-x^2}{4x^2}w_1^2-\frac{1-x^2}{4x^2}\rho^2\right)^s_+\right]}{\left(w_1^2+\rho^2\right)^{\frac{n+sp}{2}}}\rho^{n-2}\mathrm{d}\rho\\
&+\quad C\int\limits_{|w_1|\geq\frac{2x\epsilon}{1-x^2}}\mathrm{d}w_1\int_{0 }^{\infty}\frac{\left|1-\left(1+w_1-\frac{1-x^2}{4x^2}w_1^2-\frac{1-x^2}{4x^2}\rho^2\right)^s_+\right|^{p-2}\left[1-\left(1+w_1-\frac{1-x^2}{4x^2}w_1^2-\frac{1-x^2}{4x^2}\rho^2\right)^s_+\right]}{\left(w_1^2+\rho^2\right)^{\frac{n+sp}{2}}}\rho^{n-2}\mathrm{d}\rho\\
=: \ &J_1+J_2.
\end{aligned}
\end{equation*}

Step 2. In this part, we will prove $\lim\limits_{\epsilon\rightarrow0}J_1=0.$ We still omit the constant for simplicity.
\begin{equation*}
\begin{aligned}
J_1&=\int_{-\frac{2x\epsilon}{1-x^2}}^{\frac{2x\epsilon}{1-x^2}}\mathrm{d}w_1\int_{\sqrt{\frac{4x^2\epsilon^2}{(1-x^2)^2}-w_1^2} }^{\infty}\frac{\left|1-\left(1+w_1-\frac{1-x^2}{4x^2}w_1^2-\frac{1-x^2}{4x^2}\rho^2\right)^s_+\right|^{p-2}\left[1-\left(1+w_1-\frac{1-x^2}{4x^2}w_1^2-\frac{1-x^2}{4x^2}\rho^2\right)^s_+\right]}{\left(w_1^2+\rho^2\right)^{\frac{n+sp}{2}}}\rho^{n-2}\mathrm{d}\rho\\
&=\int_{-\frac{2x\epsilon}{1-x^2}}^{\frac{2x\epsilon}{1-x^2}}\mathrm{d}w_1\int_{\sqrt{\frac{4x^2\epsilon^2}{(1-x^2)^2}-w_1^2} }^{\sqrt{\frac{2x^2}{1-x^2}}}\frac{\left|1-\left(1+w_1-\frac{1-x^2}{4x^2}w_1^2-\frac{1-x^2}{4x^2}\rho^2\right)^s_{+}\right|^{p-2}\left[1-\left(1+w_1-\frac{1-x^2}{4x^2}w_1^2-\frac{1-x^2}{4x^2}\rho^2\right)^s_{+}\right]}{\left(w_1^2+\rho^2\right)^{\frac{n+sp}{2}}}\rho^{n-2}\mathrm{d}\rho\\
&\quad+\int_{-\frac{2x\epsilon}{1-x^2}}^{\frac{2x\epsilon}{1-x^2}}\mathrm{d}w_1\int_{\sqrt{\frac{2x^2}{1-x^2}} }^{\infty}\frac{\left|1-\left(1+w_1-\frac{1-x^2}{4x^2}w_1^2-\frac{1-x^2}{4x^2}\rho^2\right)^s_+\right|^{p-2}\left[1-\left(1+w_1-\frac{1-x^2}{4x^2}w_1^2-\frac{1-x^2}{4x^2}\rho^2\right)^s_+\right]}{\left(w_1^2+\rho^2\right)^{\frac{n+sp}{2}}}\rho^{n-2}\mathrm{d}\rho\\
&=:J_{11}+J_{12}.
\end{aligned}
\end{equation*}
Firstly, we can readily estimate  $J_{12}$,
\begin{equation*}
\begin{aligned}
\lvert J_{12}\rvert&\leq\int_{-\frac{2x\epsilon}{1-x^2}}^{\frac{2x\epsilon}{1-x^2}}\mathrm{d}w_1\int_{\sqrt{\frac{2x^2}{1-x^2}} }^{\infty}\frac{\left[1-\left(1+w_1-\frac{1-x^2}{4x^2}w_1^2-\frac{1-x^2}{4x^2}\rho^2\right)^s_+\right]^{p-1}}{\left(w_1^2+\rho^2\right)^{\frac{n+sp}{2}}}\rho^{n-2}\mathrm{d}\rho\\
&\leq\int_{-\frac{2x\epsilon}{1-x^2}}^{\frac{2x\epsilon}{1-x^2}}\mathrm{d}w_1\int_{\sqrt{\frac{2x^2}{1-x^2}} }^{\infty}\frac{\rho^{n-2}}{\left(w_1^2+\rho^2\right)^{\frac{n+sp}{2}}}\mathrm{d}\rho\\
&\leq\int_{-\frac{2x\epsilon}{1-x^2}}^{\frac{2x\epsilon}{1-x^2}}\mathrm{d}w_1\int_{\sqrt{\frac{2x^2}{1-x^2}} }^{\infty}\frac{1}{\rho^{2+sp}}\mathrm{d}\rho
\leq C(x)\int_{-\frac{2x\epsilon}{1-x^2}}^{\frac{2x\epsilon}{1-x^2}}\mathrm{d}w_1\rightarrow 0 \quad \text{as} \quad \epsilon\rightarrow 0.
\end{aligned}
\end{equation*}
Secondly, we estimate $J_{11}$ by taylor expansion,
\begin{equation*}
\begin{aligned}
&\quad J_{11}\\
&=-\int_{0}^{\frac{2x\epsilon}{1-x^2}}\mathrm{d}w_1\int_{\sqrt{\frac{4x^2\epsilon^2}{(1-x^2)^2}-w_1^2} }^{\sqrt{\frac{4x^2}{(1-x^2)^2}w_1-w_1^2}}\frac{\left[\left(1+w_1-\frac{1-x^2}{4x^2}w_1^2-\frac{1-x^2}{4x^2}\rho^2\right)^s-1\right]^{p-1}}{\left(w_1^2+\rho^2\right)^{\frac{n+sp}{2}}}\rho^{n-2}\mathrm{d}\rho\\
&\quad+\int_{0}^{\frac{2x\epsilon}{1-x^2}}\mathrm{d}w_1\int_{\sqrt{\frac{4x^2}{(1-x^2)^2}w_1-w_1^2} }^{\sqrt{\frac{2x^2}{1-x^2}}}\frac{\left[1-\left(1+w_1-\frac{1-x^2}{4x^2}w_1^2-\frac{1-x^2}{4x^2}\rho^2\right)^s\right]^{p-1}}{\left(w_1^2+\rho^2\right)^{\frac{n+sp}{2}}}\rho^{n-2}\mathrm{d}\rho\\
&\quad+\int_{0}^{\frac{2x\epsilon}{1-x^2}}\mathrm{d}w_1\int_{\sqrt{\frac{4x^2\epsilon^2}{(1-x^2)^2}-w_1^2} }^{\sqrt{\frac{2x^2}{1-x^2}}}\frac{\left[1-\left(1-w_1-\frac{1-x^2}{4x^2}w_1^2-\frac{1-x^2}{4x^2}\rho^2\right)^s\right]^{p-1}}{\left(w_1^2+\rho^2\right)^{\frac{n+sp}{2}}}\rho^{n-2}\mathrm{d}\rho\\
&=\int_{0}^{\frac{2x\epsilon}{1-x^2}}\mathrm{d}w_1\int_{\sqrt{\frac{4x^2\epsilon^2}{(1-x^2)^2}-w_1^2} }^{\sqrt{\frac{4x^2}{(1-x^2)^2}w_1-w_1^2}}\frac{\left[1\!-\!\left(1\!-\!w_1\!-\!\frac{1-\!x^2}{4x^2}w_1^2\!-\!\frac{1\!-\!x^2}{4x^2}\rho^2\right)^s\right]^{p-1}\!-\!\left[\left(1\!+\!w_1\!-\!\frac{1\!-\!x^2}{4x^2}w_1^2\!-\!\frac{1\!-\!x^2}{4x^2}\rho^2\right)^s\!-\!1\right]^{p-1}}{\left(w_1^2\!+\!\rho^2\right)^{\frac{n+sp}{2}}}\rho^{n-2}\mathrm{d}\rho\\
&\quad+\int_{0}^{\frac{2x\epsilon}{1-x^2}}\mathrm{d}w_1\int_{\sqrt{\frac{4x^2}{(1-x^2)^2}w_1-w_1^2} }^{\sqrt{\frac{2x^2}{1-x^2}}}\frac{\left[1-\left(1+w_1-\frac{1-x^2}{4x^2}w_1^2-\frac{1-x^2}{4x^2}\rho^2\right)^s\right]^{p-1}}{\left(w_1^2+\rho^2\right)^{\frac{n+sp}{2}}}\rho^{n-2}\mathrm{d}\rho\\
&\quad+\int_{0}^{\frac{2x\epsilon}{1-x^2}}\mathrm{d}w_1\int_{\sqrt{\frac{4x^2}{(1-x^2)^2}w_1-w_1^2} }^{\sqrt{\frac{2x^2}{1-x^2}}}\frac{\left[1-\left(1-w_1-\frac{1-x^2}{4x^2}w_1^2-\frac{1-x^2}{4x^2}\rho^2\right)^s\right]^{p-1}}{\left(w_1^2+\rho^2\right)^{\frac{n+sp}{2}}}\rho^{n-2}\mathrm{d}\rho\\
&\leq \int_{0}^{\frac{2x\epsilon}{1-x^2}}\mathrm{d}w_1\int_{\sqrt{\frac{4x^2\epsilon^2}{(1-x^2)^2}-w_1^2} }^{\sqrt{\frac{4x^2}{(1-x^2)^2}w_1-w_1^2}}\frac{Cw_1^{p-2}\left[2\!-\!\left(1\!-\!w_1\!-\!\frac{1\!-\!x^2}{4x^2}w_1^2\!-\!\frac{1\!-\!x^2}{4x^2}\rho^2\right)^s\!-\!\left(1\!+\!w_1\!-\!\frac{1\!-\!x^2}{4x^2}w_1^2\!-\!\frac{1\!-\!x^2}{4x^2}\rho^2\right)^s\right]}{\left(w_1^2\!+\!\rho^2\right)^{\frac{n+sp}{2}}}\rho^{n-2}\mathrm{d}\rho\\
&\quad+2\int_{0}^{\frac{2x\epsilon}{1-x^2}}\mathrm{d}w_1\int_{\sqrt{\frac{4x^2}{(1-x^2)^2}w_1-w_1^2} }^{\sqrt{\frac{2x^2}{1-x^2}}}\frac{C\left[\frac{1-x^2}{4x^2}\rho^2\right]^{p-1}}{\left(w_1^2+\rho^2\right)^{\frac{n+sp}{2}}}\rho^{n-2}\mathrm{d}\rho\\
&\leq \int_{0}^{\frac{2x\epsilon}{1\!-\!x^2}}\mathrm{d}w_1\int_{\sqrt{\frac{4x^2\epsilon^2}{(1\!-\!x^2)^2}-\!w_1^2} }^{\sqrt{\frac{4x^2}{(1-\!x^2)^2}w_1\!-w_1^2}}\frac{Cw_1^{p-\!2}\left[s\frac{1\!-\!x^2}{2x^2}w_1^2\!+\!O(w_1^2)\right]}{\left(w_1^2\!+\!\rho^2\right)^{\frac{n+sp}{2}}}\rho^{n-2}\mathrm{d}\rho\!\\
&\quad+\!2C\int_{0}^{\frac{2x\epsilon}{1\!-\!x^2}}\mathrm{d}w_1\int_{\sqrt{\frac{4x^2}{(1\!-\!x^2)^2}w_1\!-\!w_1^2} }^{\sqrt{\frac{2x^2}{1\!-\!x^2}}}\frac{\left[\frac{1\!-\!x^2}{4x^2}\rho^2\right]^{p-\!1}}{\left(w_1^2\!+\!\rho^2\right)^{\frac{n+\!sp}{2}}}\rho^{n-\!2}\mathrm{d}\rho.\\
\end{aligned}
\end{equation*}
Moreover, we have 
\begin{equation*}
\begin{aligned}
&\quad\int_{0}^{\frac{2x\epsilon}{1-x^2}}\mathrm{d}w_1\int_{\sqrt{\frac{4x^2\epsilon^2}{(1-x^2)^2}-w_1^2} }^{\sqrt{\frac{4x^2}{(1-x^2)^2}w_1-w_1^2}}\frac{Cw_1^{p-2}\left[s\frac{1-x^2}{2x^2}w_1^2+O(w_1^2)\right]}{\left(w_1^2+\rho^2\right)^{\frac{n+sp}{2}}}\rho^{n-2}\mathrm{d}\rho\\
&\leq \int_{0}^{\frac{2x\epsilon}{1-x^2}}\mathrm{d}w_1\int_{\sqrt{\frac{4x^2\epsilon^2}{(1-x^2)^2}-w_1^2} }^{\sqrt{\frac{4x^2}{(1-x^2)^2}w_1-w_1^2}}\frac{Cw_1^{p}}{\left(w_1^2+\rho^2\right)^{\frac{n+sp}{2}}}\rho^{n-2}\mathrm{d}\rho\\
&\leq \int_{0}^{\frac{2x\epsilon}{1-x^2}}\mathrm{d}w_1\int_{\sqrt{\frac{4x^2\epsilon^2}{(1-x^2)^2}-w_1^2} }^{\sqrt{\frac{4x^2}{(1-x^2)^2}w_1-w_1^2}}\frac{Cw_1^{p}}{\left(w_1^2+\rho^2\right)w_1^{sp}}\mathrm{d}\rho\\
&=C\int_{0}^{\frac{2x\epsilon}{1-x^2}}w_1^{p-sp-1}\mathrm{d}w_1\int_{\sqrt{\frac{4x^2\epsilon^2}{(1-x^2)^2}-w_1^2} }^{\sqrt{\frac{4x^2}{(1-x^2)^2}w_1-w_1^2}}\frac{w_1}{\left(w_1^2+\rho^2\right)}\mathrm{d}\rho\\
&=C\int_{0}^{\frac{2x\epsilon}{1-x^2}}w_1^{p-sp-1}\mathrm{d}w_1 \Big\{\arctan\frac{\rho}{w_1}\Big\}_{\sqrt{\frac{4x^2\epsilon^2}{(1-x^2)^2}-w_1^2}}^{\sqrt{\frac{4x^2}{(1-x^2)^2}w_1-w_1^2}}\\
&\leq C \int_{0}^{\frac{2x\epsilon}{1-x^2}}w_1^{p-sp-1}\mathrm{d}w_1\rightarrow 0 \quad \text{as} \quad \epsilon\rightarrow 0.
\end{aligned}
\end{equation*}
And 
\begin{equation*}
\begin{aligned}
&\quad\int_{0}^{\frac{2x\epsilon}{1-x^2}}\mathrm{d}w_1\int_{\sqrt{\frac{4x^2}{(1-x^2)^2}w_1-w_1^2} }^{\sqrt{\frac{2x^2}{1-x^2}}}\frac{\left[\frac{1-x^2}{4x^2}\rho^2\right]^{p-1}}{\left(w_1^2+\rho^2\right)^{\frac{n+sp}{2}}}\rho^{n-2}\mathrm{d}\rho\\
&\leq C\int_{0}^{\frac{2x\epsilon}{1-x^2}}\mathrm{d}w_1\int_{\sqrt{\frac{4x^2}{(1-x^2)^2}w_1-w_1^2} }^{\sqrt{\frac{2x^2}{1-x^2}}}\frac{\rho^{2p+n-4}}{\left(w_1^2+\rho^2\right)^{\frac{n+sp}{2}}}\mathrm{d}\rho\\
&\leq 
\begin{cases}
C\int_{0}^{\frac{2x\epsilon}{1-x^2}}\mathrm{d}w_1\int_{\sqrt{\frac{4x^2}{(1-x^2)^2}w_1-w_1^2} }^{\sqrt{\frac{2x^2}{1-x^2}}}\rho^{2p-sp-4}\mathrm{d}\rho\leq C\int_{0}^{\frac{2x\epsilon}{1-x^2}}\mathrm{d}w_1\rightarrow 0 \quad \text{as} \quad \epsilon\rightarrow 0, & \text{if}\quad 2p-sp-3>0;\\
C\int_{0}^{\frac{2x\epsilon}{1-x^2}}\mathrm{d}w_1\int_{\sqrt{\frac{4x^2}{(1-x^2)^2}w_1-w_1^2} }^{\sqrt{\frac{2x^2}{1-x^2}}}\frac{1}{(w_1^2+\rho^2)^{\frac{4-(2p-sp)}{2}}}\mathrm{d}\rho\\
\leq C \int_{0}^{\frac{2x\epsilon}{1-x^2}}\mathrm{d}w_1\frac{\sqrt{\frac{2x^2}{1-x^2}}}{\left[\frac{4x^2}{(1-x^2)^2}w_1\right]^{\frac{4-(2p-sp)}{2}}}\leq C\int_{0}^{\frac{2x\epsilon}{1-x^2}}w_1^{\frac{2p-sp-4}{2}}\mathrm{d}w_1\rightarrow 0 \quad \text{as} \quad \epsilon\rightarrow 0, & \text{if}\quad 2p-sp-3\leq0.\\
\end{cases}
\end{aligned}
\end{equation*}
Hence $$J_1\rightarrow 0 \quad \text{as} \quad \epsilon\rightarrow 0.$$

Step 3. We work on $J_2$ partly,
\begin{equation*}
\begin{aligned}
J_2&=\int\limits_{|w_1|\geq\frac{2x\epsilon}{1-x^2}}\mathrm{d}w_1\int_{0 }^{\infty}\frac{\left|1-\left(1+w_1-\frac{1-x^2}{4x^2}w_1^2-\frac{1-x^2}{4x^2}\rho^2\right)^s_+\right|^{p-2}\left[1-\left(1+w_1-\frac{1-x^2}{4x^2}w_1^2-\frac{1-x^2}{4x^2}\rho^2\right)^s_+\right]}{\left(w_1^2+\rho^2\right)^{\frac{n+sp}{2}}}\rho^{n-2}\mathrm{d}\rho\\
&=\int_{\frac{2x\epsilon}{1-x^2}}^{\infty}\mathrm{d}w_1\int_{0 }^{\infty}\frac{\left|1-\left(1+w_1-\frac{1-x^2}{4x^2}w_1^2-\frac{1-x^2}{4x^2}\rho^2\right)^s_+\right|^{p-2}\left[1-\left(1+w_1-\frac{1-x^2}{4x^2}w_1^2-\frac{1-x^2}{4x^2}\rho^2\right)^s_+\right]}{\left(w_1^2+\rho^2\right)^{\frac{n+sp}{2}}}\rho^{n-2}\mathrm{d}\rho\\
&\quad+\int_{\frac{2x\epsilon}{1-x^2}}^{\infty}\mathrm{d}w_1\int_{0 }^{\infty}\frac{\left[1-\left(1-w_1-\frac{1-x^2}{4x^2}w_1^2-\frac{1-x^2}{4x^2}\rho^2\right)^s_+\right]^{p-1}}{\left(w_1^2+\rho^2\right)^{\frac{n+sp}{2}}}\rho^{n-2}\mathrm{d}\rho.\\
\end{aligned}
\end{equation*}
Set $\displaystyle y=\frac{\rho}{w_1},$ then 
\begin{equation*}
\begin{aligned}
J_2
&=\int_{\frac{2x\epsilon}{1-x^2}}^{\infty}\frac{1}{w_1^{1+sp}}\mathrm{d}w_1\int_{0 }^{\infty}\frac{\left|1-\left(1+w_1-\frac{1-x^2}{4x^2}(1+y^2)w_1^2\right)^s_+\right|^{p-2}\left[1-\left(1+w_1-\frac{1-x^2}{4x^2}(1+y^2)w_1^2\right)^s_+\right]}{\left(1+y^2\right)^{\frac{n+sp}{2}}}y^{n-2}\mathrm{d}y\\
&\quad+\int_{\frac{2x\epsilon}{1-x^2}}^{\infty}\frac{1}{w_1^{1+sp}}\mathrm{d}w_1\int_{0 }^{\infty}\frac{\left[1-\left(1-w_1-\frac{1-x^2}{4x^2}(1+y^2)w_1^2\right)^s_+\right]^{p-1}}{\left(1+y^2\right)^{\frac{n+sp}{2}}}y^{n-2}\mathrm{d}y\\
&=:J_{21}+J_{22},
\end{aligned}
\end{equation*}
where 
\begin{equation*}
\begin{aligned}
J_{22}
&=\int_{\frac{2x\epsilon}{1-x^2}}^{\infty}\frac{1}{w_1^{1+sp}}\mathrm{d}w_1\int_{0 }^{\infty}\frac{\left[1-\left(1-w_1-\frac{1-x^2}{4x^2}(1+y^2)w_1^2\right)^s_+\right]^{p-1}}{\left(1+y^2\right)^{\frac{n+sp}{2}}}y^{n-2}\mathrm{d}y\\
&=\int_{\frac{2x\epsilon}{1-x^2}}^{\frac{2x}{1+x}}\frac{1}{w_1^{1+sp}}\mathrm{d}w_1\int_{0 }^{\sqrt{\frac{4x^2}{1-x^2}\frac{1-w_1}{w_1^2}}-1}\frac{\left[1-\left(1-w_1-\frac{1-x^2}{4x^2}(1+y^2)w_1^2\right)^s\right]^{p-1}}{\left(1+y^2\right)^{\frac{n+sp}{2}}}y^{n-2}\mathrm{d}y\\
&\quad+\int_{\frac{2x\epsilon}{1-x^2}}^{\frac{2x}{1+x}}\frac{1}{w_1^{1+sp}}\mathrm{d}w_1\int_{\sqrt{\frac{4x^2}{1-x^2}\frac{1-w_1}{w_1^2}}-1}^{\infty}\frac{y^{n-2}}{\left(1+y^2\right)^{\frac{n+sp}{2}}}\mathrm{d}y+\int_{\frac{2x}{1+x}}^{\infty}\frac{1}{w_1^{1+sp}}\mathrm{d}w_1\int_{0}^{\infty}\frac{y^{n-2}}{\left(1+y^2\right)^{\frac{n+sp}{2}}}\mathrm{d}y\\
&=:\textcircled{1}+\textcircled{2}+\textcircled{3}.
\end{aligned}
\end{equation*}
And
\begin{equation*}
\begin{aligned}
J_{21}
&=\int_{\frac{2x\epsilon}{1-x^2}}^{\infty}\frac{1}{w_1^{1+sp}}\mathrm{d}w_1\int_{0 }^{\infty}\frac{\left|1-\left(1+w_1-\frac{1-x^2}{4x^2}(1+y^2)w_1^2\right)^s_+\right|^{p-2}\left[1-\left(1+w_1-\frac{1-x^2}{4x^2}(1+y^2)w_1^2\right)^s_+\right]}{\left(1+y^2\right)^{\frac{n+sp}{2}}}y^{n-2}\mathrm{d}y\\
&=-\int_{\frac{2x\epsilon}{1-x^2}}^{\frac{4x^2}{1-x^2}}\frac{1}{w_1^{1+sp}}\int_{0 }^{\sqrt{\frac{4x^2}{1-x^2}\frac{1}{w_1}-1}}\frac{\left[\left(1+w_1-\frac{1-x^2}{4x^2}(1+y^2)w_1^2\right)^s-1\right]^{p-1}}{\left(1+y^2\right)^{\frac{n+sp}{2}}}y^{n-2}\mathrm{d}y\mathrm{d}w_1\\
&\quad+\int_{\frac{2x\epsilon}{1-x^2}}^{\frac{4x^2}{1-x^2}}\frac{1}{w_1^{1+sp}}\int_{\sqrt{\frac{4x^2}{1-x^2}\frac{1}{w_1}-1}}^{\sqrt{\frac{4x^2}{1-x^2}\frac{1+w_1}{w_1^2}-1}}\frac{\left[1-\left(1+w_1-\frac{1-x^2}{4x^2}(1+y^2)w_1^2\right)^s\right]^{p-1}}{\left(1+y^2\right)^{\frac{n+sp}{2}}}y^{n-2}\mathrm{d}y\mathrm{d}w_1\\
&\quad+\int_{\frac{2x\epsilon}{1-x^2}}^{\frac{4x^2}{1-x^2}}\frac{1}{w_1^{1+sp}}\int_{\sqrt{\frac{4x^2}{1-x^2}\frac{1+w_1}{w_1^2}-1}}^{\infty}\frac{y^{n-2}}{\left(1+y^2\right)^{\frac{n+sp}{2}}}\mathrm{d}y\mathrm{d}w_1\\
&\quad+\int_{\frac{4x^2}{1-x^2}}^{\frac{2x}{1-x}}\frac{1}{w_1^{1+sp}}\int_{0}^{\sqrt{\frac{4x^2}{1-x^2}\frac{1+w_1}{w_1^2}-1}}\frac{\left[1-\left(1+w_1-\frac{1-x^2}{4x^2}(1+y^2)w_1^2\right)^s\right]^{p-1}}{\left(1+y^2\right)^{\frac{n+sp}{2}}}y^{n-2}\mathrm{d}y\mathrm{d}w_1\\
&\quad+\int_{\frac{4x^2}{1-x^2}}^{\frac{2x}{1-x}}\frac{1}{w_1^{1+sp}}\int_{\sqrt{\frac{4x^2}{1-x^2}\frac{1+w_1}{w_1^2}-1}}^{\infty}\frac{y^{n-2}}{\left(1+y^2\right)^{\frac{n+sp}{2}}}\mathrm{d}y\mathrm{d}w_1\\
&\quad+\int_{\frac{2x}{1-x}}^{\infty}\frac{1}{w_1^{1+sp}}\int_{0}^{\infty}\frac{y^{n-2}}{\left(1+y^2\right)^{\frac{n+sp}{2}}}\mathrm{d}y\mathrm{d}w_1\\
&=:\textcircled{1}'+\textcircled{2}'+\textcircled{3}'+\textcircled{4}'+\textcircled{5}'+\textcircled{6}'.
\end{aligned}
\end{equation*}

Step 4. In this step, we will figure out $\textcircled{1}+\textcircled{1}'$,
\begin{equation*}
\begin{aligned}
&\textcircled{1}+\textcircled{1}'\\
= \ &\int_{\frac{2x\epsilon}{1-x^2}}^{\frac{1}{2}}\frac{1}{w_1^{1+sp}}\int_{0 }^{\sqrt{\frac{4x^2}{1-x^2}\frac{1}{w_1}-1}}\frac{\left|1-\left(1+w_1-\frac{1-x^2}{4x^2}(1+y^2)w_1^2\right)^s\right|^{p-2}\left[1-\left(1+w_1-\frac{1-x^2}{4x^2}(1+y^2)w_1^2\right)^s\right]}{\left(1+y^2\right)^{\frac{n+sp}{2}}}y^{n-2}\mathrm{d}y\mathrm{d}w_1\\
&+\int_{\frac{2x\epsilon}{1-x^2}}^{\frac{1}{2}}\frac{1}{w_1^{1+sp}}\int_{\sqrt{\frac{4x^2}{1-x^2}\frac{1}{w_1}-1} }^{\sqrt{\frac{4x^2}{1-x^2}\frac{1-w_1}{w_1^2}-1}}\frac{\left[1-\left(1-w_1-\frac{1-x^2}{4x^2}(1+y^2)w_1^2\right)^s\right]^{p-1}}{\left(1+y^2\right)^{\frac{n+sp}{2}}}y^{n-2}\mathrm{d}y\mathrm{d}w_1\\
&+\int_{\frac{1}{2}}^{\frac{2x}{1+x}}\frac{1}{w_1^{1+sp}}\int_{0 }^{\sqrt{\frac{4x^2}{1-x^2}\frac{1-w_1}{w_1^2}-1}}\frac{\left[1-\left(1-w_1-\frac{1-x^2}{4x^2}(1+y^2)w_1^2\right)^s\right]^{p-1}}{\left(1+y^2\right)^{\frac{n+sp}{2}}}y^{n-2}\mathrm{d}y\mathrm{d}w_1\\
&-\int_{\frac{1}{2}}^{\frac{4x^2}{1-x^2}}\frac{1}{w_1^{1+sp}}\int_{0 }^{\sqrt{\frac{4x^2}{1-x^2}\frac{1}{w_1}-1}}\frac{\left[\left(1+w_1-\frac{1-x^2}{4x^2}(1+y^2)w_1^2\right)^s-1\right]^{p-1}}{\left(1+y^2\right)^{\frac{n+sp}{2}}}y^{n-2}\mathrm{d}y\mathrm{d}w_1\\
=: \ &\left(\uppercase\expandafter{\romannumeral1}\right)+\left( \uppercase\expandafter{\romannumeral2}\right) +\left( \uppercase\expandafter{\romannumeral3}\right) -\left( \uppercase\expandafter{\romannumeral4}\right).
\end{aligned}
\end{equation*}
At the moment, there are 11 terms: $\textcircled{2},$ $\textcircled{3},$ $\textcircled{2}',$ $\textcircled{3}',$ $\textcircled{4}',$ $\textcircled{5}',$ $\textcircled{6}',$ $\left(\uppercase\expandafter{\romannumeral1}\right),$ $\left( \uppercase\expandafter{\romannumeral2}\right),$ $\left( \uppercase\expandafter{\romannumeral3}\right),$ $\left( \uppercase\expandafter{\romannumeral4}\right).$

Step 5. From now on, we will estimate item by item. Given \eqref{zuichuguji}, in this step, we claim that $\textcircled{4}', \textcircled{5}', \textcircled{6}'$ are uniformly bounded when multiplied by $(1-x)^{-s}$ when $x$ is close to $1$.
\begin{equation*}
\begin{aligned}
(1-x)^{-s}\left(\textcircled{4}'+\textcircled{5}'\right)&\leq C(1-x)^{-s} \int_{\frac{4x^2}{1-x^2}}^{\frac{2x}{1-x}}\frac{1}{w_1^{1+sp}}\int_{0}^{\infty}\frac{y^{n-2}}{\left(1+y^2\right)^{\frac{n+sp}{2}}}\mathrm{d}y\mathrm{d}w_1\\
&\leq C(1-x)^{-s} \int_{\frac{4x^2}{1-x^2}}^{\frac{2x}{1-x}}\frac{1}{w_1^{1+sp}}\mathrm{d}w_1\\
&\leq C(1-x)^{-s}\left(\frac{1-x^2}{4x^2}\right)^{sp}\\
&\leq C(1-x)^{s(p-1)}.
\end{aligned}
\end{equation*}
And
\begin{equation*}
\begin{aligned}
(1-x)^{-s}\textcircled{6}'&\leq C(1-x)^{-s} \int_{\frac{2x}{1-x}}^{\infty}\frac{1}{w_1^{1+sp}}\int_{0}^{\infty}\frac{y^{n-2}}{\left(1+y^2\right)^{\frac{n+sp}{2}}}\mathrm{d}y\mathrm{d}w_1\\
&\leq C(1-x)^{-s} \int_{\frac{4x^2}{1-x^2}}^{\frac{2x}{1-x}}\frac{1}{w_1^{1+sp}}\mathrm{d}w_1\\
&\leq C(1-x)^{-s}\left(\frac{1-x}{2x}\right)^{sp}\\
&\leq C(1-x)^{s(p-1)}.
\end{aligned}
\end{equation*}

Step 6. We assert that $\textcircled{3}$ can be replaced by 
\begin{equation}
\int_{1}^{\infty}\frac{1}{w_1^{1+sp}}\mathrm{d}w_1\int_{0}^{\infty}\frac{y^{n-2}}{\left(1+y^2\right)^{\frac{n+sp}{2}}}\mathrm{d}y=\frac{1}{sp}\int_{0}^{\infty}\frac{y^{n-2}}{\left(1+y^2\right)^{\frac{n+sp}{2}}}\mathrm{d}y.
\end{equation}
This is because 
\begin{equation*}
\begin{aligned}
&\quad(1-x)^{-s}\int_{\frac{2x}{1+x}}^{1}\frac{1}{w_1^{1+sp}}\mathrm{d}w_1\int_{0}^{\infty}\frac{y^{n-2}}{\left(1+y^2\right)^{\frac{n+sp}{2}}}\mathrm{d}y\\
&\leq C(1-x)^{-s}\int_{\frac{2x}{1+x}}^{1}\frac{1}{w_1^{1+sp}}\mathrm{d}w_1\\
&=C(1-x)^{-s}\left(\frac{1+x}{2x}\right)^{sp}\left[1-\left(1-\frac{1-x}{1+x}\right)^{sp}\right]\\
&\leq C(1-x)^{-s}\left(\frac{1+x}{2x}\right)^{sp}\left[C\cdot sp\frac{1-x}{1+x}\right]\\
&\leq C(1-x)^{1-s}.
\end{aligned}
\end{equation*}

Step 7. We will reformulate  $\left( \uppercase\expandafter{\romannumeral3}\right)$ as 
\begin{equation*}
\int_{\frac{1}{2}}^{1}\frac{\left[1-\left(1-w_1\right)^s\right]^{p-1}}{w_1^{1+sp}}\mathrm{d}w_1\int_{0}^{\infty}\frac{y^{n-2}}{\left(1+y^2\right)^{\frac{n+sp}{2}}}\mathrm{d}y.
\end{equation*}
Firstly, we will substitute $\left( \uppercase\expandafter{\romannumeral3}\right)$  by
\begin{equation*}
\int_{\frac{1}{2}}^{\frac{2x}{1+x}}\frac{1}{w_1^{1+sp}}\int_{0}^{\sqrt{\frac{4x^2}{1-x^2}\frac{1-w_1}{w_1^2}-1}}\frac{\left[1-\left(1-w_1\right)^s\right]^{p-1}}{\left(1+y^2\right)^{\frac{n+sp}{2}}}y^{n-2}\mathrm{d}y\mathrm{d}w_1.
\end{equation*}
Since 
\begin{equation*}
\begin{aligned}
&\quad(1-x)^{-s}\int_{\frac{1}{2}}^{\frac{2x}{1+x}}\frac{1}{w_1^{1+sp}}\int_{0}^{\sqrt{\frac{4x^2}{1-x^2}\frac{1-w_1}{w_1^2}-1}}\frac{\left[1-\left(1-w_1-\frac{1-x^2}{4x^2}(1+y^2)w_1^2\right)^s\right]^{p-1}\!-\!\left[1-\left(1-w_1\right)^s\right]^{p-1}}{\left(1+y^2\right)^{\frac{n+sp}{2}}}y^{n-2}\mathrm{d}y\mathrm{d}w_1\\
&\leq C(1-x)^{-s}\int_{\frac{1}{2}}^{\frac{2x}{1+x}}\frac{1}{w_1^{1+sp}}\int_{0}^{\sqrt{\frac{4x^2}{1-x^2}\frac{1-w_1}{w_1^2}-1}}\frac{w_1^{s(p-2)}\left[\left(1-w_1\right)^s-\left(1-w_1-\frac{1-x^2}{4x^2}(1+y^2)w_1^2\right)^s\right]}{\left(1+y^2\right)^{\frac{n+sp}{2}}}y^{n-2}\mathrm{d}y\mathrm{d}w_1\\
&\leq C(1-x)^{-s}\int_{\frac{1}{2}}^{\frac{2x}{1+x}}\frac{1}{w_1^{1+2s}}\int_{0}^{\sqrt{\frac{4x^2}{1-x^2}\frac{1-w_1}{w_1^2}-1}}\frac{\left( \frac{1-x^2}{4x^2}\right) ^s(1+y^2)^sw_1^{2s}}{\left(1+y^2\right)^{\frac{n+sp}{2}}}y^{n-2}\mathrm{d}y\mathrm{d}w_1\\
&\leq C \int_{\frac{1}{2}}^{\frac{2x}{1+x}}\frac{1}{w_1}\int_{0}^{\sqrt{\frac{4x^2}{1-x^2}\frac{1-w_1}{w_1^2}-1}}\frac{1}{\left(1+y^2\right)^{\frac{2+s(p-2)}{2}}}\mathrm{d}y\mathrm{d}w_1\leq C\int_{\frac{1}{2}}^{\frac{2x}{1+x}}\frac{1}{w_1}\mathrm{d}w_1\leq C.\\
\end{aligned}
\end{equation*}
Moreover, $\left( \uppercase\expandafter{\romannumeral3}\right)$ can be replaced by
\begin{equation*}
\int_{\frac{1}{2}}^{\frac{2x}{1+x}}\frac{1}{w_1^{1+sp}}\int_{0}^{\infty}\frac{\left[1-\left(1-w_1\right)^s\right]^{p-1}}{\left(1+y^2\right)^{\frac{n+sp}{2}}}y^{n-2}\mathrm{d}y\mathrm{d}w_1.
\end{equation*}
Due to the fact that 
\begin{equation*}
\begin{aligned}
&\quad(1-x)^{-s}\int_{\frac{1}{2}}^{\frac{2x}{1+x}}\frac{1}{w_1^{1+sp}}\int_{\sqrt{\frac{4x^2}{1-x^2}\frac{1-w_1}{w_1^2}-1}}^{\infty}\frac{\left[1-\left(1-w_1\right)^s\right]^{p-1}}{\left(1+y^2\right)^{\frac{n+sp}{2}}}y^{n-2}\mathrm{d}y\mathrm{d}w_1\\
&\leq C(1-x)^{-s}\int_{\frac{1}{2}}^{\frac{2x}{1+x}}\frac{1}{w_1^{1+sp}}\int_{\sqrt{\frac{4x^2}{1-x^2}\frac{1-w_1}{w_1^2}-1}}^{\infty}\frac{y^{n-2}}{\left(1+y^2\right)^{\frac{n+sp}{2}}}\mathrm{d}y\mathrm{d}w_1\\
&=C(1-x)^{-s}\int_{\frac{1}{2}}^{\frac{2x}{1+x}}\frac{1}{w_1^{1+sp}}\int_{\sqrt{\frac{\left(\frac{2x}{1+x}-w_1\right)\left(w_1+\frac{2x}{1-x}\right)}{w_1^2}}}^{\infty}\frac{y^{n-2}}{\left(1+y^2\right)^{\frac{n+sp}{2}}}\mathrm{d}y\mathrm{d}w_1\\
&\leq C(1-x)^{-s}\int_{\frac{1}{2}}^{\frac{2x}{1+x}}\frac{1}{w_1^{1+sp}}\int_{\sqrt{\frac{\left(\frac{2x}{1+x}-w_1\right)\left(w_1+\frac{2x}{1-x}\right)}{w_1^2}}}^{\infty}\frac{1}{y^{1+2s}\left(1+y^2\right)^{\frac{1+s(p-2)}{2}}}\mathrm{d}y\mathrm{d}w_1\\
&\leq C(1-x)^{-s}\int_{\frac{1}{2}}^{\frac{2x}{1+x}}\frac{1}{w_1^{1+sp}}\int_{\sqrt{\frac{\left(\frac{2x}{1+x}-w_1\right)\left(w_1+\frac{2x}{1-x}\right)}{w_1^2}}}^{\infty}\frac{1}{y^{1+2s}}\mathrm{d}y\mathrm{d}w_1\\
&\leq C(1-x)^{-s}\int_{\frac{1}{2}}^{\frac{2x}{1+x}}\frac{1}{w_1^{1+sp}}{\frac{w_1^{2s}}{\left(\frac{2x}{1+x}-w_1\right)^{s}\left(w_1+\frac{2x}{1-x}\right)^s}}\mathrm{d}w_1\\
&\leq C\int_{\frac{1}{2}}^{\frac{2x}{1+x}}\left( \frac{2x}{1+x}-w_1\right)^{-s}\mathrm{d}w_1\leq C.
\end{aligned}
\end{equation*}
Furthermore,   we transform $\left( \uppercase\expandafter{\romannumeral3}\right)$ into
\begin{equation*}
\int_{\frac{1}{2}}^{1}\frac{1}{w_1^{1+sp}}\int_{0}^{\infty}\frac{\left[1-\left(1-w_1\right)^s\right]^{p-1}}{\left(1+y^2\right)^{\frac{n+sp}{2}}}y^{n-2}\mathrm{d}y\mathrm{d}w_1.
\end{equation*}
This is because
\begin{equation*}
\begin{aligned}
&\quad(1-x)^{-s}\int_{\frac{2x}{1+x}}^1\frac{1}{w_1^{1+sp}}\int_{0}^{\infty}\frac{\left[1-\left(1-w_1\right)^s\right]^{p-1}}{\left(1+y^2\right)^{\frac{n+sp}{2}}}y^{n-2}\mathrm{d}y\mathrm{d}w_1\\
&\leq C(1-x)^{-s}\int_{\frac{2x}{1+x}}^1\frac{1}{w_1^{1+sp}}\mathrm{d}w_1\\
&\leq C(1-x)^{-s}\left(\frac{1+x}{2x}\right)^{sp}\left[1-\left(1-\frac{1-x}{1+x}\right)^{sp}\right]\\
&\leq C(1-x)^{-s}\left(\frac{1+x}{2x}\right)^{sp}\left[C\cdot sp\frac{1-x}{1+x}\right]\leq C(1-x)^{1-s}.
\end{aligned}
\end{equation*}

Step 8. In this part,  we are going to replace $\left( \uppercase\expandafter{\romannumeral1}\right)$
by 
\begin{equation}
\int_{0}^{\frac{1}{2}}\frac{\left[1-\left(1-w_1\right)^s\right]^{p-1}-\left[\left(1+w_1\right)^s-1\right]^{p-1}}{w_1^{1+sp}}\mathrm{d}w_1\int_{0}^{\infty}\frac{y^{n-2}}{\left(1+y^2\right)^{\frac{n+sp}{2}}}\mathrm{d}y.
\end{equation}
Firstly, we will reduce 
$\left( \uppercase\expandafter{\romannumeral1}\right)$
to 
\begin{equation*}
\int_{\frac{2x\epsilon}{1-x}}^{\frac{1}{2}}\frac{1}{w_1^{1+sp}}\int_{0}^{\sqrt{\frac{4x^2}{1-x^2}\frac{1}{w_1}-1}}\frac{\left[1-\left(1-w_1\right)^s\right]^{p-1}-\left[\left(1+w_1\right)^s-1\right]^{p-1}}{\left(1+y^2\right)^{\frac{n+sp}{2}}}y^{n-2}\mathrm{d}y\mathrm{d}w_1.
\end{equation*}
On the one hand, 
\begin{equation*}
\begin{aligned}
&\quad(1-x)^{-s}\int_{\frac{2x\epsilon}{1-x}}^{\frac{1}{2}}\frac{1}{w_1^{1+sp}}\int_{0}^{\sqrt{\frac{4x^2}{1-x^2}\frac{1}{w_1}-1}}\frac{\left[1-\left(1-w_1-\frac{1-x^2}{4x^2}(1+y^2)w_1^2\right)^s\right]^{p-1}\!-\!\left[1-\left(1-w_1\right)^s\right]^{p-1}}{\left(1+y^2\right)^{\frac{n+sp}{2}}}y^{n-2}\mathrm{d}y\mathrm{d}w_1\\
&= (1-x)^{-s}\int_{\frac{2x\epsilon}{1-x}}^{\frac{1}{4}}\frac{1}{w_1^{1+sp}}\int_{0}^{\sqrt{\frac{4x^2}{1-x^2}\frac{1}{w_1}-1}}\frac{\left[1-\left(1-w_1-\frac{1-x^2}{4x^2}(1+y^2)w_1^2\right)^s\right]^{p-1}\!-\!\left[1-\left(1-w_1\right)^s\right]^{p-1}}{\left(1+y^2\right)^{\frac{n+sp}{2}}}y^{n-2}\mathrm{d}y\mathrm{d}w_1\\
&\quad+(1-x)^{-s}\int_{\frac{1}{4}}^{\frac{1}{2}}\frac{1}{w_1^{1+sp}}\int_{0}^{\sqrt{\frac{4x^2}{1-x^2}\frac{1}{w_1}-1}}\frac{\left[1-\left(1-w_1-\frac{1-x^2}{4x^2}(1+y^2)w_1^2\right)^s\right]^{p-1}\!-\!\left[1-\left(1-w_1\right)^s\right]^{p-1}}{\left(1+y^2\right)^{\frac{n+sp}{2}}}y^{n-2}\mathrm{d}y\mathrm{d}w_1\\
&\leq C(1-x)^{-s}\int_{\frac{2x\epsilon}{1-x}}^{\frac{1}{4}}\frac{1}{w_1^{1+sp}}\int_{0}^{\sqrt{\frac{4x^2}{1-x^2}\frac{1}{w_1}-1}}\frac{w_1^{p-2}\left[\left(1-w_1\right)^s-\left(1-w_1-\frac{1-x^2}{4x^2}(1+y^2)w_1^2\right)^s\right]}{\left(1+y^2\right)^{\frac{n+sp}{2}}}y^{n-2}\mathrm{d}y\mathrm{d}w_1\\
&\quad+C(1-x)^{-s}\int_{\frac{1}{4}}^{\frac{1}{2}}\frac{1}{w_1^{1+sp}}\int_{0}^{\sqrt{\frac{4x^2}{1-x^2}\frac{1}{w_1}-1}}\frac{w_1^{s(p-2)}\left[\left(1-w_1\right)^s-\left(1-w_1-\frac{1-x^2}{4x^2}(1+y^2)w_1^2\right)^s\right]}{\left(1+y^2\right)^{\frac{n+sp}{2}}}y^{n-2}\mathrm{d}y\mathrm{d}w_1\\
&\leq C(1-x)^{-s}\int_{\frac{2x\epsilon}{1-x}}^{\frac{1}{4}}\frac{1}{w_1^{1+sp}}\int_{0}^{\sqrt{\frac{4x^2}{1-x^2}\frac{1}{w_1}-1}}\frac{w_1^{p-2}\left(1-w_1\right)^s\left[1-\left(1-\frac{1-x^2}{4x^2}(1+y^2)\frac{w_1^2}{1-w_1}\right)^s\right]}{\left(1+y^2\right)^{\frac{n+sp}{2}}}y^{n-2}\mathrm{d}y\mathrm{d}w_1\\
&\quad+C(1-x)^{-s}\int_{\frac{1}{4}}^{\frac{1}{2}}\frac{1}{w_1^{1+sp}}\int_{0}^{\sqrt{\frac{4x^2}{1-x^2}\frac{1}{w_1}-1}}\frac{w_1^{s(p-2)}\left( \frac{1-x^2}{4x^2}\right) ^s(1+y^2)^sw_1^{2s}}{\left(1+y^2\right)^{\frac{n+sp}{2}}}y^{n-2}\mathrm{d}y\mathrm{d}w_1.
\end{aligned}
\end{equation*}
Further,
\begin{equation*}
\begin{aligned}
&\quad(1-x)^{-s}\int_{\frac{2x\epsilon}{1-x}}^{\frac{1}{4}}\frac{1}{w_1^{1+sp}}\int_{0}^{\sqrt{\frac{4x^2}{1-x^2}\frac{1}{w_1}-1}}\frac{w_1^{p-2}\left(1-w_1\right)^s\left[1-\left(1-\frac{1-x^2}{4x^2}(1+y^2)\frac{w_1^2}{1-w_1}\right)^s\right]}{\left(1+y^2\right)^{\frac{n+sp}{2}}}y^{n-2}\mathrm{d}y\mathrm{d}w_1\\
&\leq C(1-x)^{-s}\int_{\frac{2x\epsilon}{1-x}}^{\frac{1}{4}}\frac{1}{w_1^{1+sp}}\int_{0}^{\sqrt{\frac{4x^2}{1-x^2}\frac{1}{w_1}-1}}\frac{w_1^{p-2}\left(1-w_1\right)^s\left[s\frac{1-x^2}{4x^2}(1+y^2)\frac{w_1^2}{1-w_1}\right]}{\left(1+y^2\right)^{\frac{n+sp}{2}}}y^{n-2}\mathrm{d}y\mathrm{d}w_1\\
&\leq C(1-x)^{1-s}\int_{\frac{2x\epsilon}{1-x}}^{\frac{1}{4}}w_1^{p-sp-1}\int_{0}^{\sqrt{\frac{4x^2}{1-x^2}\frac{1}{w_1}-1}}\frac{1}{\left(1+y^2\right)^{\frac{sp}{2}}}\mathrm{d}y\mathrm{d}w_1\\
&\leq\begin{cases}
C(1-x)^{1-s}\int_{\frac{2x\epsilon}{1-x}}^{\frac{1}{4}}w_1^{p-sp-1}\mathrm{d}w_1\leq C(1-x)^{1-s}, & \text{if} \quad sp>2;\\
\\
C(1-x)^{1-s}\int_{\frac{2x\epsilon}{1-x}}^{\frac{1}{4}}w_1^{p-sp-1}\left(\frac{4x^2}{1-x^2}\frac{1}{w_1}\right)^{1-\frac{sp}{2}}\mathrm{d}w_1
\\\quad \leq C(1-x)^{\frac{s(p-2)}{2}}\int_{\frac{2x\epsilon}{1-x}}^{\frac{1}{4}}w_1^{p-sp/2-2}\mathrm{d}w_1\leq C(1-x)^{\frac{s(p-2)}{2}},&  \text{if} \quad sp\leq 2.
\end{cases}
\end{aligned}
\end{equation*}
And
\begin{equation*}
\begin{aligned}
&\quad(1-x)^{-s}\int_{\frac{1}{4}}^{\frac{1}{2}}\frac{1}{w_1^{1+sp}}\int_{0}^{\sqrt{\frac{4x^2}{1-x^2}\frac{1}{w_1}-1}}\frac{w_1^{s(p-2)}\left( \frac{1-x^2}{4x^2}\right) ^s(1+y^2)^sw_1^{2s}}{\left(1+y^2\right)^{\frac{n+sp}{2}}}y^{n-2}\mathrm{d}y\mathrm{d}w_1\\
&\leq C(1-x)^{-s}\int_{\frac{1}{4}}^{\frac{1}{2}}\frac{1}{w_1}\int_{0}^{\sqrt{\frac{4x^2}{1-x^2}\frac{1}{w_1}-1}}\frac{\left( \frac{1-x^2}{4x^2}\right) ^s(1+y^2)^s}{\left(1+y^2\right)^{\frac{n+sp}{2}}}y^{n-2}\mathrm{d}y\mathrm{d}w_1\\
&\leq C \int_{\frac{1}{4}}^{\frac{1}{2}}\frac{1}{w_1}\int_{0}^{\sqrt{\frac{4x^2}{1-x^2}\frac{1}{w_1}-1}}\frac{1}{\left(1+y^2\right)^{\frac{2+s(p-2)}{2}}}\mathrm{d}y\mathrm{d}w_1\leq C.\\
\end{aligned}
\end{equation*}
On the other hand, we have 
\begin{equation*}
\begin{aligned}
&\quad(1-x)^{-s}\int_{\frac{2x\epsilon}{1-x}}^{\frac{1}{2}}\frac{1}{w_1^{1+sp}}\int_{0}^{\sqrt{\frac{4x^2}{1-x^2}\frac{1}{w_1}-1}}\frac{\left[\left(1+w_1\right)^s-1\right]^{p-1}\!-\!\left[\left(1+w_1-\frac{1-x^2}{4x^2}(1+y^2)w_1^2\right)^s-1\right]^{p-1}}{\left(1+y^2\right)^{\frac{n+sp}{2}}}y^{n-2}\mathrm{d}y\mathrm{d}w_1\\
&\leq C(1-x)^{-s}\int_{\frac{2x\epsilon}{1-x}}^{\frac{1}{2}}\frac{1}{w_1^{1+sp}}\int_{0}^{\sqrt{\frac{4x^2}{1-x^2}\frac{1}{w_1}-1}}\frac{w_1^{p-2}\left[\left(1+w_1\right)^s-\left(1+w_1-\frac{1-x^2}{4x^2}(1+y^2)w_1^2\right)^s\right]}{\left(1+y^2\right)^{\frac{n+sp}{2}}}y^{n-2}\mathrm{d}y\mathrm{d}w_1\\
&=C(1-x)^{-s}\int_{\frac{2x\epsilon}{1-x}}^{\frac{1}{2}}\frac{1}{w_1^{1+sp}}\int_{0}^{\sqrt{\frac{4x^2}{1-x^2}\frac{1}{w_1}-1}}\frac{w_1^{p-2}\left(1+w_1\right)^s\left[1-\left(1-\frac{1-x^2}{4x^2}(1+y^2)\frac{w_1^2}{1+w_1}\right)^s\right]}{\left(1+y^2\right)^{\frac{n+sp}{2}}}y^{n-2}\mathrm{d}y\mathrm{d}w_1\\
&\leq C(1-x)^{-s}\int_{\frac{2x\epsilon}{1-x}}^{\frac{1}{2}}\frac{1}{w_1^{1+sp}}\int_{0}^{\sqrt{\frac{4x^2}{1-x^2}\frac{1}{w_1}-1}}\frac{w_1^{p-2}\left(1+w_1\right)^s\frac{1-x^2}{4x^2}(1+y^2)\frac{w_1^2}{1+w_1}}{\left(1+y^2\right)^{\frac{n+sp}{2}}}y^{n-2}\mathrm{d}y\mathrm{d}w_1\\
&\leq C(1-x)^{1-s}\int_{\frac{2x\epsilon}{1-x}}^{\frac{1}{2}}w_1^{p-sp-1}\int_{0}^{\sqrt{\frac{4x^2}{1-x^2}\frac{1}{w_1}-1}}\frac{1}{\left(1+y^2\right)^{\frac{sp}{2}}}\mathrm{d}y\mathrm{d}w_1\\
&\leq\begin{cases}
C(1-x)^{1-s}, & \text{if} \quad sp>2;\\
C(1-x)^{\frac{s(p-2)}{2}}, & \text{if} \quad sp\leq 2.
\end{cases}
\end{aligned}
\end{equation*}
Moreover, we derive from the above arguments that $\left( \uppercase\expandafter{\romannumeral1}\right)$
can be replaced by 
\begin{equation*}
\int_{0}^{\frac{1}{2}}\frac{1}{w_1^{1+sp}}\int_{0}^{\sqrt{\frac{4x^2}{1-x^2}\frac{1}{w_1}-1}}\frac{\left[1-\left(1-w_1\right)^s\right]^{p-1}-\left[\left(1+w_1\right)^s-1\right]^{p-1}}{\left(1+y^2\right)^{\frac{n+sp}{2}}}y^{n-2}\mathrm{d}y\mathrm{d}w_1.
\end{equation*} 
Furthermore, we conclude this step by the following estimate.
\begin{equation*}
\begin{aligned}
&\quad(1-x)^{-s}\int_{0}^{\frac{1}{2}}\frac{1}{w_1^{1+sp}}\int_{\sqrt{\frac{4x^2}{1-x^2}\frac{1}{w_1}-1}}^{\infty}\frac{\left[1-\left(1-w_1\right)^s\right]^{p-1}-\left[\left(1+w_1\right)^s-1\right]^{p-1}}{\left(1+y^2\right)^{\frac{n+sp}{2}}}y^{n-2}\mathrm{d}y\mathrm{d}w_1\\
&=(1-x)^{-s}\int_{0}^{\frac{1}{2}}\frac{\left[1-\left(1-w_1\right)^s\right]^{p-1}-\left[\left(1+w_1\right)^s-1\right]^{p-1}}{w_1^{1+sp}}\int_{\sqrt{\frac{4x^2}{1-x^2}\frac{1}{w_1}-1}}^{\infty}\frac{1}{\left(1+y^2\right)^{\frac{n+sp}{2}}}y^{n-2}\mathrm{d}y\mathrm{d}w_1\\
&\leq C(1-x)^{-s}\int_{0}^{\frac{1}{2}}\frac{\left[1-\left(1-w_1\right)^s\right]^{p-1}-\left[\left(1+w_1\right)^s-1\right]^{p-1}}{w_1^{1+sp}}\int_{\sqrt{\frac{4x^2}{1-x^2}\frac{1}{w_1}-1}}^{\infty}\frac{1}{y^{2+sp}}\mathrm{d}y\mathrm{d}w_1\\
&\leq C(1-x)^{-s}\int_{0}^{\frac{1}{2}}\frac{w_1^{p-2}\left[2-\left(1-w_1\right)^s-\left(1+w_1\right)^s\right]}{w_1^{1+sp}}\frac{1}{\left[\frac{4x^2}{1-x^2}\frac{1}{w_1}\right]^{\frac{1+sp}{2}}}\mathrm{d}w_1\\
&\leq C(1-x)^{\frac{1+s(p-2)}{2}}\int_{0}^{\frac{1}{2}}w_1^{\frac{2p-sp-1}{2}}\mathrm{d}w_1\leq C(1-x)^{\frac{1+s(p-2)}{2}}.\\
\end{aligned}
\end{equation*}

Step 9. We demonstrate that  $(1-x)^{-s}\textcircled{2}$ is uniformly  bounded when $x$ is close to $1$.
\begin{equation*}
\begin{aligned}
&\quad(1-x)^{-s}\int_{\frac{2x\epsilon}{1-x^2}}^{\frac{2x}{1+x}}\frac{1}{w_1^{1+sp}}\int_{\sqrt{\frac{4x^2}{1-x^2}\frac{1-w_1}{w_1^2}}-1}^{\infty}\frac{y^{n-2}}{\left(1+y^2\right)^{\frac{n+sp}{2}}}\mathrm{d}y\mathrm{d}w_1\\
&=(1-x)^{-s}\int_{\frac{2x\epsilon}{1-x^2}}^{\frac{2x}{1+x}}\frac{1}{w_1^{1+sp}}\int_{\sqrt{\frac{\left(\frac{2x}{1+x}-w_1\right)\left(w_1+\frac{2x}{1-x}\right)}{w_1^2}}}^{\infty}\frac{y^{n-2}}{\left(1+y^2\right)^{\frac{n+sp}{2}}}\mathrm{d}y\mathrm{d}w_1\\
&=(1-x)^{-s}\int_{\frac{2x\epsilon}{1-x^2}}^{\frac{x}{1+x}}\frac{1}{w_1^{1+sp}}\int_{\sqrt{\frac{\left(\frac{2x}{1+x}-w_1\right)\left(w_1+\frac{2x}{1-x}\right)}{w_1^2}}}^{\infty}\frac{y^{n-2}}{\left(1+y^2\right)^{\frac{n+sp}{2}}}\mathrm{d}y\mathrm{d}w_1\\
&\quad+(1-x)^{-s}\int_{\frac{x}{1+x}}^{\frac{2x}{1+x}}\frac{1}{w_1^{1+sp}}\int_{\sqrt{\frac{\left(\frac{2x}{1+x}-w_1\right)\left(w_1+\frac{2x}{1-x}\right)}{w_1^2}}}^{\infty}\frac{y^{n-2}}{\left(1+y^2\right)^{\frac{n+sp}{2}}}\mathrm{d}y\mathrm{d}w_1.\\
\end{aligned}
\end{equation*}
For one thing, 
\begin{equation*}
\begin{aligned}
&\quad(1-x)^{-s}\int_{\frac{2x\epsilon}{1-x^2}}^{\frac{x}{1+x}}\frac{1}{w_1^{1+sp}}\int_{\sqrt{\frac{\left(\frac{2x}{1+x}-w_1\right)\left(w_1+\frac{2x}{1-x}\right)}{w_1^2}}}^{\infty}\frac{y^{n-2}}{\left(1+y^2\right)^{\frac{n+sp}{2}}}\mathrm{d}y\mathrm{d}w_1\\
&\leq C (1-x)^{-s}\int_{\frac{2x\epsilon}{1-x^2}}^{\frac{x}{1+x}}\frac{1}{w_1^{1+sp}}\frac{1}{\left[\frac{\left(\frac{2x}{1+x}-w_1\right)\left(w_1+\frac{2x}{1-x}\right)}{w_1^2}\right]^{\frac{1+sp}{2}}}\mathrm{d}w_1\\
&\leq C(1-x)^{\frac{1+s(p-2)}{2}}.
\end{aligned}
\end{equation*}
For another thing, 
\begin{equation*}
\begin{aligned}
&\quad(1-x)^{-s}\int_{\frac{x}{1+x}}^{\frac{2x}{1+x}}\frac{1}{w_1^{1+sp}}\int_{\sqrt{\frac{\left(\frac{2x}{1+x}-w_1\right)\left(w_1+\frac{2x}{1-x}\right)}{w_1^2}}}^{\infty}\frac{y^{n-2}}{\left(1+y^2\right)^{\frac{n+sp}{2}}}\mathrm{d}y\mathrm{d}w_1\\
&\leq C (1-x)^{-s}\int_{\frac{x}{1+x}}^{\frac{2x}{1+x}}\frac{1}{w_1^{1+sp}}\int_{\sqrt{\frac{\left(\frac{2x}{1+x}-w_1\right)\left(w_1+\frac{2x}{1-x}\right)}{w_1^2}}}^{\infty}\frac{1}{y^{1+2s}\left(1+y^2\right)^{\frac{1+s(p-2)}{2}}}\mathrm{d}y\mathrm{d}w_1\\
&\leq C (1-x)^{-s}\int_{\frac{x}{1+x}}^{\frac{2x}{1+x}}\frac{1}{w_1^{1+sp}}\int_{\sqrt{\frac{\left(\frac{2x}{1+x}-w_1\right)\left(w_1+\frac{2x}{1-x}\right)}{w_1^2}}}^{\infty}\frac{1}{y^{1+2s}}\mathrm{d}y\mathrm{d}w_1\\
&\leq C (1-x)^{-s}\int_{\frac{x}{1+x}}^{\frac{2x}{1+x}}\frac{1}{w_1^{1+sp}}\frac{1}{\left[\frac{\left(\frac{2x}{1+x}-w_1\right)\left(w_1+\frac{2x}{1-x}\right)}{w_1^2}\right]^s}\\
&\leq C\int_{\frac{x}{1+x}}^{\frac{2x}{1+x}}\left(\frac{2x}{1+x}-w_1\right)^{-s}\leq C.
\end{aligned}
\end{equation*}

Step 10. We will prove that   $(1-x)^{-s}\textcircled{3}'$ is uniformly  bounded when $x$ is close to $1$.
\begin{equation*}
\begin{aligned}
&\quad(1-x)^{-s}\int_{\frac{2x\epsilon}{1-x^2}}^{\frac{4x^2}{1-x^2}}\frac{1}{w_1^{1+sp}}\int_{\sqrt{\frac{4x^2}{1-x^2}\frac{1+w_1}{w_1^2}-1}}^{\infty}\frac{y^{n-2}}{\left(1+y^2\right)^{\frac{n+sp}{2}}}\mathrm{d}y\mathrm{d}w_1\\
&=(1-x)^{-s}\int_{\frac{2x\epsilon}{1-x^2}}^{\frac{4x^2}{1-x^2}}\frac{1}{w_1^{1+sp}}\int_{\sqrt{\frac{\left(\frac{2x}{1-x}-w_1 \right) \left( w_1+\frac{2x}{1+x}\right) }{w_1^2}}}^{\infty}\frac{y^{n-2}}{\left(1+y^2\right)^{\frac{n+sp}{2}}}\mathrm{d}y\mathrm{d}w_1\\
&=(1-x)^{-s}\int_{\frac{2x\epsilon}{1-x^2}}^{\frac{2x^2}{1-x^2}}\frac{1}{w_1^{1+sp}}\int_{\sqrt{\frac{\left(\frac{2x}{1-x}-w_1 \right) \left( w_1+\frac{2x}{1+x}\right) }{w_1^2}}}^{\infty}\frac{y^{n-2}}{\left(1+y^2\right)^{\frac{n+sp}{2}}}\mathrm{d}y\mathrm{d}w_1\\
&\quad+(1-x)^{-s}\int_{\frac{2x^2}{1-x^2}}^{\frac{4x^2}{1-x^2}}\frac{1}{w_1^{1+sp}}\int_{\sqrt{\frac{\left(\frac{2x}{1-x}-w_1 \right) \left( w_1+\frac{2x}{1+x}\right) }{w_1^2}}}^{\infty}\frac{y^{n-2}}{\left(1+y^2\right)^{\frac{n+sp}{2}}}\mathrm{d}y\mathrm{d}w_1.\\
\end{aligned}
\end{equation*}
On the one hand, 
\begin{equation*}
\begin{aligned}
&\quad(1-x)^{-s}\int_{\frac{2x\epsilon}{1-x^2}}^{\frac{2x^2}{1-x^2}}\frac{1}{w_1^{1+sp}}\int_{\sqrt{\frac{\left(\frac{2x}{1-x}-w_1 \right) \left( w_1+\frac{2x}{1+x}\right) }{w_1^2}}}^{\infty}\frac{y^{n-2}}{\left(1+y^2\right)^{\frac{n+sp}{2}}}\mathrm{d}y\mathrm{d}w_1\\
&\leq C(1-x)^{-s}\int_{\frac{2x\epsilon}{1-x^2}}^{\frac{2x^2}{1-x^2}}\frac{1}{w_1^{1+sp}}\frac{1}{\left[\frac{\left(\frac{2x}{1-x}-w_1\right)\left(w_1+\frac{2x}{1+x}\right)}{w_1^2}\right]^{\frac{1+sp}{2}}}\mathrm{d}w_1\\
&\leq C(1-x)^{\frac{1+s(p-2)}{2}} \int_{\frac{2x\epsilon}{1-x^2}}^{\frac{2x^2}{1-x^2}}\left(w_1+\frac{2x}{1+x}\right)^{-\frac{1+sp}{2}}\mathrm{d}w_1\\
&\leq\begin{cases}
C(1-x)^{\frac{1+s(p-2)}{2}},& \text{if} \quad sp>1;\\
C(1-x)^{\frac{1+s(p-2)}{2}}\log\frac{1}{1-x},& \text{if} \quad sp=1;\\
C(1-x)^{\frac{1+s(p-2)}{2}}\frac{1}{(1-x)^{\frac{1-sp}{2}}}=C(1-x)^{s(p-1)},& \text{if} \quad sp<1.\\
\end{cases}
\end{aligned}
\end{equation*}
On the other hand,
\begin{equation*}
\begin{aligned}
&\quad(1-x)^{-s}\int_{\frac{2x^2}{1-x^2}}^{\frac{4x^2}{1-x^2}}\frac{1}{w_1^{1+sp}}\int_{\sqrt{\frac{\left(\frac{2x}{1-x}-w_1 \right) \left( w_1+\frac{2x}{1+x}\right) }{w_1^2}}}^{\infty}\frac{y^{n-2}}{\left(1+y^2\right)^{\frac{n+sp}{2}}}\mathrm{d}y\mathrm{d}w_1\\
&\leq C(1-x)^{-s}\int_{\frac{2x^2}{1-x^2}}^{\frac{4x^2}{1-x^2}}\frac{1}{w_1^{1+sp}}\mathrm{d}w_1\leq C(1-x)^{s(p-1)}.
\end{aligned}
\end{equation*}

Step 11. This part is intended to prove $(1-x)^{-s}\textcircled{2}'$ is uniformly bounded when $x$ is close to $1$.
\begin{equation*}
\begin{aligned}
&\quad(1-x)^{-s}\int_{\frac{2x\epsilon}{1-x^2}}^{\frac{4x^2}{1-x^2}}\frac{1}{w_1^{1+sp}}\int_{\sqrt{\frac{4x^2}{1-x^2}\frac{1}{w_1}-1}}^{\sqrt{\frac{4x^2}{1-x^2}\frac{1+w_1}{w_1^2}-1}}\frac{\left[1-\left(1+w_1-\frac{1-x^2}{4x^2}(1+y^2)w_1^2\right)^s\right]^{p-1}}{\left(1+y^2\right)^{\frac{n+sp}{2}}}y^{n-2}\mathrm{d}y\mathrm{d}w_1\\
&=(1-x)^{-s}\int_{\frac{2x\epsilon}{1-x^2}}^{\frac{4x^2}{1-x^2}}\frac{1}{w_1^{1+sp}}\int_{\sqrt{\frac{4x^2}{1-x^2}\frac{1}{w_1}-1}}^{\sqrt{\frac{4x^2}{1-x^2}\frac{\frac{1}{2}+w_1}{w_1^2}-1}}\frac{\left[1-\left(1+w_1-\frac{1-x^2}{4x^2}(1+y^2)w_1^2\right)^s\right]^{p-1}}{\left(1+y^2\right)^{\frac{n+sp}{2}}}y^{n-2}\mathrm{d}y\mathrm{d}w_1\\
&\quad+(1-x)^{-s}\int_{\frac{2x\epsilon}{1-x^2}}^{\frac{4x^2}{1-x^2}}\frac{1}{w_1^{1+sp}}\int_{\sqrt{\frac{4x^2}{1-x^2}\frac{\frac{1}{2}+w_1}{w_1^2}-1}}^{\sqrt{\frac{4x^2}{1-x^2}\frac{1+w_1}{w_1^2}-1}}\frac{\left[1-\left(1+w_1-\frac{1-x^2}{4x^2}(1+y^2)w_1^2\right)^s\right]^{p-1}}{\left(1+y^2\right)^{\frac{n+sp}{2}}}y^{n-2}\mathrm{d}y\mathrm{d}w_1.\\
\end{aligned}
\end{equation*}
We will estimate the two terms separately.
\begin{equation*}
\begin{aligned}
&\quad(1-x)^{-s}\int_{\frac{2x\epsilon}{1-x^2}}^{\frac{4x^2}{1-x^2}}\frac{1}{w_1^{1+sp}}\int_{\sqrt{\frac{4x^2}{1-x^2}\frac{1}{w_1}-1}}^{\sqrt{\frac{4x^2}{1-x^2}\frac{\frac{1}{2}+w_1}{w_1^2}-1}}\frac{\left[1-\left(1+w_1-\frac{1-x^2}{4x^2}(1+y^2)w_1^2\right)^s\right]^{p-1}}{\left(1+y^2\right)^{\frac{n+sp}{2}}}y^{n-2}\mathrm{d}y\mathrm{d}w_1\\
&=(1-x)^{-s}\int_{\frac{2x\epsilon}{1-x^2}}^{1}\frac{1}{w_1^{1+sp}}\int_{\sqrt{\frac{4x^2}{1-x^2}\frac{1}{w_1}-1}}^{\sqrt{\frac{4x^2}{1-x^2}\frac{\frac{1}{2}+w_1}{w_1^2}-1}}\frac{\left[1-\left(1+w_1-\frac{1-x^2}{4x^2}(1+y^2)w_1^2\right)^s\right]^{p-1}}{\left(1+y^2\right)^{\frac{n+sp}{2}}}y^{n-2}\mathrm{d}y\mathrm{d}w_1\\
&\quad+(1-x)^{-s}\int_{1}^{\frac{2x^2}{1-x^2}}\frac{1}{w_1^{1+sp}}\int_{\sqrt{\frac{4x^2}{1-x^2}\frac{1}{w_1}-1}}^{\sqrt{\frac{4x^2}{1-x^2}\frac{\frac{1}{2}+w_1}{w_1^2}-1}}\frac{\left[1-\left(1+w_1-\frac{1-x^2}{4x^2}(1+y^2)w_1^2\right)^s\right]^{p-1}}{\left(1+y^2\right)^{\frac{n+sp}{2}}}y^{n-2}\mathrm{d}y\mathrm{d}w_1\\
&\quad+(1-x)^{-s}\int_{\frac{2x^2}{1-x^2}}^{\frac{4x^2}{1-x^2}}\frac{1}{w_1^{1+sp}}\int_{\sqrt{\frac{4x^2}{1-x^2}\frac{1}{w_1}-1}}^{\sqrt{\frac{4x^2}{1-x^2}\frac{\frac{1}{2}+w_1}{w_1^2}-1}}\frac{\left[1-\left(1+w_1-\frac{1-x^2}{4x^2}(1+y^2)w_1^2\right)^s\right]^{p-1}}{\left(1+y^2\right)^{\frac{n+sp}{2}}}y^{n-2}\mathrm{d}y\mathrm{d}w_1\\
&\leq C(1-x)^{-s}\int_{\frac{2x\epsilon}{1-x^2}}^{1}\frac{1}{w_1^{1+sp}}\int_{\sqrt{\frac{4x^2}{1-x^2}\frac{1}{w_1}-1}}^{\sqrt{\frac{4x^2}{1-x^2}\frac{\frac{1}{2}+w_1}{w_1^2}-1}}\frac{\left[\frac{1-x^2}{4x^2}(1+y^2)w_1^2\right]^{p-1}}{\left(1+y^2\right)^{\frac{n+sp}{2}}}y^{n-2}\mathrm{d}y\mathrm{d}w_1\\
&\quad+C(1-x)^{-s}\int_{1}^{\frac{2x^2}{1-x^2}}\frac{1}{w_1^{1+sp}}\frac{1}{\left[\frac{4x^2}{1-x^2}\frac{1}{w_1}\right]^{\frac{1+sp}{2}}}\mathrm{d}w_1+C(1-x)^{s(p-1)}.
\end{aligned}
\end{equation*}
Moreover,
\begin{equation*}
\begin{aligned}
&\quad(1-x)^{-s}\int_{\frac{2x\epsilon}{1-x^2}}^{1}\frac{1}{w_1^{1+sp}}\int_{\sqrt{\frac{4x^2}{1-x^2}\frac{1}{w_1}-1}}^{\sqrt{\frac{4x^2}{1-x^2}\frac{\frac{1}{2}+w_1}{w_1^2}-1}}\frac{\left[\frac{1-x^2}{4x^2}(1+y^2)w_1^2\right]^{p-1}}{\left(1+y^2\right)^{\frac{n+sp}{2}}}y^{n-2}\mathrm{d}y\mathrm{d}w_1\\
&\leq C(1-x)^{p-s-1}\int_{\frac{2x\epsilon}{1-x^2}}^{1}w_1^{2p-sp-3}\int_{\sqrt{\frac{4x^2}{1-x^2}\frac{1}{w_1}-1}}^{\sqrt{\frac{4x^2}{1-x^2}\frac{\frac{1}{2}+w_1}{w_1^2}-1}}(1+y^2)^{p-\frac{sp}{2}-2}\mathrm{d}y\mathrm{d}w_1\\
&\leq \begin{cases}
C(1-x)^{p-s-1}\int_{\frac{2x\epsilon}{1-x^2}}^{1}w_1^{2p-sp-3}\left[\frac{4x^2}{1-x^2}\frac{\frac{1}{2}+w_1}{w_1^2}\right]^{p-\frac{sp}{2}-3/2}\mathrm{d}w_1\\\quad
\leq C(1-x)^{\frac{1+s(p-2)}{2}}\int_{\frac{2x\epsilon}{1-x^2}}^{1}\mathrm{d}w_1\leq C(1-x)^{\frac{1+s(p-2)}{2}},&\text{if}\quad p-\frac{sp}{2}-2\geq 0;\\
\\
C(1-x)^{p-s-1}\int_{\frac{2x\epsilon}{1-x^2}}^{1}w_1^{2p-sp-3}\left[\frac{4x^2}{1-x^2}\frac{1}{w_1}\right]^{p-\frac{sp}{2}-2}\sqrt{\frac{4x^2}{1-x^2}\frac{\frac{1}{2}+w_1}{w_1^2}}\mathrm{d}w_1\\\quad
\leq C(1-x)^{\frac{1+s(p-2)}{2}}\int_{\frac{2x\epsilon}{1-x^2}}^{1}w_1^{p-\frac{sp}{2}-2}\mathrm{d}w_1\leq C(1-x)^{\frac{1+s(p-2)}{2}},&\text{if}\quad p-\frac{sp}{2}-2< 0.
\end{cases}
\end{aligned}
\end{equation*}
And
\begin{equation*}
\begin{aligned}
&\quad(1-x)^{-s}\int_{1}^{\frac{2x^2}{1-x^2}}\frac{1}{w_1^{1+sp}}\frac{1}{\left[\frac{4x^2}{1-x^2}\frac{1}{w_1}\right]^{\frac{1+sp}{2}}}\mathrm{d}w_1\\
&\leq C(1-x)^{\frac{1+s(p-2)}{2}}\int_{1}^{\frac{2x^2}{1-x^2}}\frac{1}{w_1^{\frac{1+sp}{2}}}\mathrm{d}w_1\\
&\leq\begin{cases}
C(1-x)^{\frac{1+s(p-2)}{2}},&\text{if}\quad sp>1;\\
C(1-x)^{\frac{1+s(p-2)}{2}}\log\frac{1}{1-x},&\text{if}\quad sp=1;\\
C(1-x)^{s(p-1)},&\text{if}\quad sp<1.
\end{cases}
\end{aligned}
\end{equation*}
In addition,
\begin{equation*}
\begin{aligned}
&\quad(1-x)^{-s}\int_{\frac{2x\epsilon}{1-x^2}}^{\frac{4x^2}{1-x^2}}\frac{1}{w_1^{1+sp}}\int_{\sqrt{\frac{4x^2}{1-x^2}\frac{\frac{1}{2}+w_1}{w_1^2}-1}}^{\sqrt{\frac{4x^2}{1-x^2}\frac{1+w_1}{w_1^2}-1}}\frac{\left[1-\left(1+w_1-\frac{1-x^2}{4x^2}(1+y^2)w_1^2\right)^s\right]^{p-1}}{\left(1+y^2\right)^{\frac{n+sp}{2}}}y^{n-2}\mathrm{d}y\mathrm{d}w_1\\
&=(1-x)^{-s}\int_{\frac{2x\epsilon}{1-x^2}}^{\frac{2x^2}{1-x^2}}\frac{1}{w_1^{1+sp}}\int_{\sqrt{\frac{4x^2}{1-x^2}\frac{\frac{1}{2}+w_1}{w_1^2}-1}}^{\sqrt{\frac{4x^2}{1-x^2}\frac{1+w_1}{w_1^2}-1}}\frac{\left[1-\left(1+w_1-\frac{1-x^2}{4x^2}(1+y^2)w_1^2\right)^s\right]^{p-1}}{\left(1+y^2\right)^{\frac{n+sp}{2}}}y^{n-2}\mathrm{d}y\mathrm{d}w_1\\
&\quad+(1-x)^{-s}\int_{\frac{2x^2}{1-x^2}}^{\frac{4x^2}{1-x^2}}\frac{1}{w_1^{1+sp}}\int_{\sqrt{\frac{4x^2}{1-x^2}\frac{\frac{1}{2}+w_1}{w_1^2}-1}}^{\sqrt{\frac{4x^2}{1-x^2}\frac{1+w_1}{w_1^2}-1}}\frac{\left[1-\left(1+w_1-\frac{1-x^2}{4x^2}(1+y^2)w_1^2\right)^s\right]^{p-1}}{\left(1+y^2\right)^{\frac{n+sp}{2}}}y^{n-2}\mathrm{d}y\mathrm{d}w_1\\
&\leq C(1-x)^{-s}\int_{\frac{2x\epsilon}{1-x^2}}^{\frac{2x^2}{1-x^2}}\frac{1}{w_1^{1+sp}}\frac{1}{\left[\frac{4x^2}{1-x^2}\frac{\frac{1}{2}+w_1}{w_1^2}\right]^{\frac{1+sp}{2}}}\mathrm{d}w_1+C(1-x)^{s(p-1)}.
\end{aligned}
\end{equation*}
Further,
\begin{equation*}
\begin{aligned}
&\quad (1-x)^{-s}\int_{\frac{2x\epsilon}{1-x^2}}^{\frac{2x^2}{1-x^2}}\frac{1}{w_1^{1+sp}}\frac{1}{\left[\frac{4x^2}{1-x^2}\frac{\frac{1}{2}+w_1}{w_1^2}\right]^{\frac{1+sp}{2}}}\mathrm{d}w_1\\
&\leq C(1-x)^{\frac{1+s(p-2)}{2}}\int_{\frac{2x\epsilon}{1-x^2}}^{\frac{2x^2}{1-x^2}}\left(\frac{1}{2}+w_1\right)^{-\frac{1+sp}{2}}\mathrm{d}w_1\\
&\leq\begin{cases}
C(1-x)^{\frac{1+s(p-2)}{2}},&\text{if}\quad sp>1;\\
C(1-x)^{\frac{1+s(p-2)}{2}}\log\frac{1}{1-x},&\text{if}\quad sp=1;\\
C(1-x)^{s(p-1)},&\text{if}\quad sp<1.
\end{cases}
\end{aligned}
\end{equation*}

Step 12. We will prove  $(1-x)^{-s}\left( \uppercase\expandafter{\romannumeral2}\right) $ is uniformly bounded when $x$ is close to $1$.
\begin{equation*}
\begin{aligned}
&\quad(1-x)^{-s}\int_{\frac{2x\epsilon}{1-x^2}}^{\frac{1}{2}}\frac{1}{w_1^{1+sp}}\int_{\sqrt{\frac{4x^2}{1-x^2}\frac{1}{w_1}-1} }^{\sqrt{\frac{4x^2}{1-x^2}\frac{1-w_1}{w_1^2}-1}}\frac{\left[1-\left(1-w_1-\frac{1-x^2}{4x^2}(1+y^2)w_1^2\right)^s\right]^{p-1}}{\left(1+y^2\right)^{\frac{n+sp}{2}}}y^{n-2}\mathrm{d}y\mathrm{d}w_1\\
&=(1-x)^{-s}\int_{\frac{2x\epsilon}{1-x^2}}^{\frac{1}{4}}\frac{1}{w_1^{1+sp}}\int_{\sqrt{\frac{4x^2}{1-x^2}\frac{1}{w_1}-1} }^{\sqrt{\frac{4x^2}{1-x^2}\frac{1-w_1}{w_1^2}-1}}\frac{\left[1-\left(1-w_1-\frac{1-x^2}{4x^2}(1+y^2)w_1^2\right)^s\right]^{p-1}}{\left(1+y^2\right)^{\frac{n+sp}{2}}}y^{n-2}\mathrm{d}y\mathrm{d}w_1\\
&\quad+(1-x)^{-s}\int_{\frac{1}{4}}^{\frac{1}{2}}\frac{1}{w_1^{1+sp}}\int_{\sqrt{\frac{4x^2}{1-x^2}\frac{1}{w_1}-1} }^{\sqrt{\frac{4x^2}{1-x^2}\frac{1-w_1}{w_1^2}-1}}\frac{\left[1-\left(1-w_1-\frac{1-x^2}{4x^2}(1+y^2)w_1^2\right)^s\right]^{p-1}}{\left(1+y^2\right)^{\frac{n+sp}{2}}}y^{n-2}\mathrm{d}y\mathrm{d}w_1\\
&=(1-x)^{-s}\int_{\frac{2x\epsilon}{1-x^2}}^{\frac{1}{4}}\frac{1}{w_1^{1+sp}}\int_{\sqrt{\frac{4x^2}{1-x^2}\frac{1}{w_1}-1} }^{\sqrt{\frac{4x^2}{1-x^2}\frac{\frac{1}{2}-w_1}{w_1^2}-1}}\frac{\left[1-\left(1-w_1-\frac{1-x^2}{4x^2}(1+y^2)w_1^2\right)^s\right]^{p-1}}{\left(1+y^2\right)^{\frac{n+sp}{2}}}y^{n-2}\mathrm{d}y\mathrm{d}w_1\\
&\quad+(1-x)^{-s}\int_{\frac{2x\epsilon}{1-x^2}}^{\frac{1}{4}}\frac{1}{w_1^{1+sp}}\int_{\sqrt{\frac{4x^2}{1-x^2}\frac{\frac{1}{2}-w_1}{w_1^2}-1} }^{\sqrt{\frac{4x^2}{1-x^2}\frac{1-w_1}{w_1^2}-1}}\frac{\left[1-\left(1-w_1-\frac{1-x^2}{4x^2}(1+y^2)w_1^2\right)^s\right]^{p-1}}{\left(1+y^2\right)^{\frac{n+sp}{2}}}y^{n-2}\mathrm{d}y\mathrm{d}w_1\\
&\quad+(1-x)^{-s}\int_{\frac{1}{4}}^{\frac{1}{2}}\frac{1}{w_1^{1+sp}}\int_{\sqrt{\frac{4x^2}{1-x^2}\frac{1}{w_1}-1} }^{\sqrt{\frac{4x^2}{1-x^2}\frac{1-w_1}{w_1^2}-1}}\frac{\left[1-\left(1-w_1-\frac{1-x^2}{4x^2}(1+y^2)w_1^2\right)^s\right]^{p-1}}{\left(1+y^2\right)^{\frac{n+sp}{2}}}y^{n-2}\mathrm{d}y\mathrm{d}w_1\\
&\leq C(1-x)^{-s}\int_{\frac{2x\epsilon}{1-x^2}}^{\frac{1}{4}}\frac{1}{w_1^{1+sp}}\int_{\sqrt{\frac{4x^2}{1-x^2}\frac{1}{w_1}-1} }^{\sqrt{\frac{4x^2}{1-x^2}\frac{\frac{1}{2}-w_1}{w_1^2}-1}}\frac{\left[\frac{1-x^2}{4x^2}(1+y^2)w_1^2\right]^{p-1}}{\left(1+y^2\right)^{\frac{n+sp}{2}}}y^{n-2}\mathrm{d}y\mathrm{d}w_1\\
&\quad+C(1-x)^{-s}\int_{\frac{2x\epsilon}{1-x^2}}^{\frac{1}{4}}\frac{1}{w_1^{1+sp}}\frac{1}{\left[\frac{4x^2}{1-x^2}\frac{\frac{1}{2}-w_1}{w_1^2}\right]^{\frac{1+sp}{2}}}\mathrm{d}w_1+C(1-x)^{-s}\int_{\frac{1}{4}}^{\frac{1}{2}}\frac{1}{w_1^{1+sp}}\frac{1}{\left[\frac{4x^2}{1-x^2}\frac{1}{w_1}\right]^{\frac{1+sp}{2}}}\\
&\leq C(1-x)^{p-s-1}\int_{\frac{2x\epsilon}{1-x^2}}^{\frac{1}{4}}w_1^{2p-sp-3}\int_{\sqrt{\frac{4x^2}{1-x^2}\frac{1}{w_1}-1} }^{\sqrt{\frac{4x^2}{1-x^2}\frac{\frac{1}{2}-w_1}{w_1^2}-1}}(1+y^2)^{p-\frac{sp}{2}-2}\mathrm{d}y\mathrm{d}w_1+C(1-x)^{\frac{1+s(p-2)}{2}}.
\end{aligned}
\end{equation*}
Furthermore,
\begin{equation*}
\begin{aligned}
&\quad(1-x)^{p-s-1}\int_{\frac{2x\epsilon}{1-x^2}}^{\frac{1}{4}}w_1^{2p-sp-3}\int_{\sqrt{\frac{4x^2}{1-x^2}\frac{1}{w_1}-1} }^{\sqrt{\frac{4x^2}{1-x^2}\frac{\frac{1}{2}-w_1}{w_1^2}-1}}(1+y^2)^{p-\frac{sp}{2}-2}\mathrm{d}y\mathrm{d}w_1\\
&\leq \begin{cases}
C(1-x)^{p-s-1}\int_{\frac{2x\epsilon}{1-x^2}}^{\frac{1}{4}}w_1^{2p-sp-3}\left( \frac{4x^2}{1-x^2}\frac{\frac{1}{2}-w_1}{w_1^2}\right)^{p-\frac{sp}{2}-3/2}\mathrm{d}w_1\leq C(1-x)^{\frac{1+s(p-2)}{2}},&\text{if}\quad p-\frac{sp}{2}-2\geq 0;\\
\\
C(1-x)^{p-s-1}\int_{\frac{2x\epsilon}{1-x^2}}^{\frac{1}{4}}w_1^{2p-sp-3}\left( \frac{4x^2}{1-x^2}\frac{1}{w_1}\right)^{p-\frac{sp}{2}-2}\sqrt{\frac{4x^2}{1-x^2}\frac{\frac{1}{2}-w_1}{w_1^2}}\mathrm{d}w_1\leq C(1-x)^{\frac{1+s(p-2)}{2}},&\text{if}\quad p-\frac{sp}{2}-2<0.\\
\end{cases}
\end{aligned}
\end{equation*}

Step 13. We finally will prove that $\left( \uppercase\expandafter{\romannumeral4}\right) $ can be reduced to 
\begin{equation}
\int_{\frac{1}{2}}^{\infty}\frac{\left[\left(1+w_1\right)^s-1\right]^{p-1}}{w_1^{1+sp}}\int_{0 }^{\infty}\frac{y^{n-2}}{\left(1+y^2\right)^{\frac{n+sp}{2}}}\mathrm{d}y\mathrm{d}w_1.
\end{equation}
Firstly, we prove that  $\left( \uppercase\expandafter{\romannumeral4}\right) $ can be replaced by 
\begin{equation*}
\int_{\frac{1}{2}}^{\frac{4x^2}{1-x^2}}\frac{1}{w_1^{1+sp}}\int_{0 }^{\sqrt{\frac{4x^2}{1-x^2}\frac{1}{w_1}-1}}\frac{\left[\left(1+w_1\right)^s-1\right]^{p-1}}{\left(1+y^2\right)^{\frac{n+sp}{2}}}y^{n-2}\mathrm{d}y\mathrm{d}w_1.
\end{equation*}
Considering the difference,
\begin{equation*}
\begin{aligned}
&\quad(1-x)^{-s}\int_{\frac{1}{2}}^{\frac{4x^2}{1-x^2}}\frac{1}{w_1^{1+sp}}\int_{0 }^{\sqrt{\frac{4x^2}{1-x^2}\frac{1}{w_1}-1}}\frac{\left[\left(1+w_1\right)^s-1\right]^{p-1}-\left[\left(1+w_1-\frac{1-x^2}{4x^2}(1+y^2)w_1^2\right)^s-1\right]^{p-1}}{\left(1+y^2\right)^{\frac{n+sp}{2}}}y^{n-2}\mathrm{d}y\mathrm{d}w_1\\
&=(1-x)^{-s}\int_{\frac{1}{2}}^{\frac{3}{4}}\frac{1}{w_1^{1+sp}}\int_{0 }^{\sqrt{\frac{4x^2}{1-x^2}\frac{1}{w_1}-1}}\frac{\left[\left(1+w_1\right)^s-1\right]^{p-1}-\left[\left(1+w_1-\frac{1-x^2}{4x^2}(1+y^2)w_1^2\right)^s-1\right]^{p-1}}{\left(1+y^2\right)^{\frac{n+sp}{2}}}y^{n-2}\mathrm{d}y\mathrm{d}w_1\\
&\quad+(1-x)^{-s}\int_{\frac{3}{4}}^{\frac{4x^2}{1-x^2}}\frac{1}{w_1^{1+sp}}\int_{0 }^{\sqrt{\frac{4x^2}{1-x^2}\frac{1}{w_1}-1}}\frac{\left[\left(1+w_1\right)^s-1\right]^{p-1}-\left[\left(1+w_1-\frac{1-x^2}{4x^2}(1+y^2)w_1^2\right)^s-1\right]^{p-1}}{\left(1+y^2\right)^{\frac{n+sp}{2}}}y^{n-2}\mathrm{d}y\mathrm{d}w_1\\
&\leq C (1-x)^{-s}\int_{\frac{1}{2}}^{\frac{3}{4}}\frac{1}{w_1^{1+sp}}\int_{0 }^{\sqrt{\frac{4x^2}{1-x^2}\frac{1}{w_1}-1}}\frac{w_1^{p-2}\left[\left(1+w_1\right)^s-\left(1+w_1-\frac{1-x^2}{4x^2}(1+y^2)w_1^2\right)^s\right]}{\left(1+y^2\right)^{\frac{n+sp}{2}}}y^{n-2}\mathrm{d}y\mathrm{d}w_1\\
&\quad+(1-x)^{-s}\int_{\frac{3}{4}}^{\frac{2x^2}{1-x^2}}\frac{1}{w_1^{1+sp}}\int_{0 }^{\sqrt{\frac{4x^2}{1-x^2}\frac{1}{w_1}-1}}\frac{\left[\left(1+w_1\right)^s-1\right]^{p-1}-\left[\left(1+w_1-\frac{1-x^2}{4x^2}(1+y^2)w_1^2\right)^s-1\right]^{p-1}}{\left(1+y^2\right)^{\frac{n+sp}{2}}}y^{n-2}\mathrm{d}y\mathrm{d}w_1\\
&\quad+C(1-x)^{s(p-1)}.
\end{aligned}
\end{equation*}
We begin to evaluate the first two terms separately,
\begin{equation*}
\begin{aligned}
&\quad(1-x)^{-s}\int_{\frac{1}{2}}^{\frac{3}{4}}\frac{1}{w_1^{1+sp}}\int_{0 }^{\sqrt{\frac{4x^2}{1-x^2}\frac{1}{w_1}-1}}\frac{w_1^{p-2}\left[\left(1+w_1\right)^s-\left(1+w_1-\frac{1-x^2}{4x^2}(1+y^2)w_1^2\right)^s\right]}{\left(1+y^2\right)^{\frac{n+sp}{2}}}y^{n-2}\mathrm{d}y\mathrm{d}w_1\\
&\leq C (1-x)^{-s}\int_{\frac{1}{2}}^{\frac{3}{4}}\frac{1}{w_1^{1+sp}}\int_{0 }^{\sqrt{\frac{4x^2}{1-x^2}\frac{1}{w_1}-1}}\frac{w_1^{p-2}\left(1+w_1\right)^s\left[1-\left(1-\frac{1-x^2}{4x^2}(1+y^2)\frac{w_1^2}{1+w_1}\right)^s\right]}{\left(1+y^2\right)^{\frac{n+sp}{2}}}y^{n-2}\mathrm{d}y\mathrm{d}w_1\\
&\leq C (1-x)^{-s}\int_{\frac{1}{2}}^{\frac{3}{4}}\frac{1}{w_1^{1+sp}}\int_{0 }^{\sqrt{\frac{4x^2}{1-x^2}\frac{1}{w_1}-1}}\frac{w_1^{p-2}\left(1+w_1\right)^s\frac{1-x^2}{4x^2}(1+y^2)\frac{w_1^2}{1+w_1}}{\left(1+y^2\right)^{\frac{n+sp}{2}}}y^{n-2}\mathrm{d}y\mathrm{d}w_1\\
&= C(1-x)^{1-s}\int_{\frac{1}{2}}^{\frac{3}{4}}\int_{0 }^{\sqrt{\frac{4x^2}{1-x^2}\frac{1}{w_1}-1}}\frac{1}{(1+y^2)^{\frac{sp}{2}}}\mathrm{d}y\mathrm{d}w_1\\
&\leq \begin{cases}
C(1-x)^{1-s},&\text{if}\quad sp\geq 2;\\
C(1-x)^{1-s}\int_{\frac{1}{2}}^{\frac{3}{4}}\left(\frac{4x^2}{1-x^2}\frac{1}{w_1}\right)^{1-\frac{sp}{2}}\mathrm{d}w_1\leq C(1-x)^{\frac{s(p-2)}{2}},&\text{if}\quad sp< 2.\\
\end{cases}
\end{aligned}
\end{equation*}
Moreover, 
\begin{equation*}
\begin{aligned}
&\quad(1-x)^{-s}\int_{\frac{3}{4}}^{\frac{2x^2}{1-x^2}}\frac{1}{w_1^{1+sp}}\int_{0 }^{\sqrt{\frac{4x^2}{1-x^2}\frac{1}{w_1}-1}}\frac{\left[\left(1+w_1\right)^s-1\right]^{p-1}-\left[\left(1+w_1-\frac{1-x^2}{4x^2}(1+y^2)w_1^2\right)^s-1\right]^{p-1}}{\left(1+y^2\right)^{\frac{n+sp}{2}}}y^{n-2}\mathrm{d}y\mathrm{d}w_1\\
&=(1-x)^{-s}\int_{\frac{3}{4}}^{\frac{2x^2}{1-x^2}}\frac{1}{w_1^{1+sp}}\int_{0 }^{\sqrt{\frac{2x^2}{1-x^2}\frac{1}{w_1}-1}}\frac{\left[\left(1+w_1\right)^s-1\right]^{p-1}-\left[\left(1+w_1-\frac{1-x^2}{4x^2}(1+y^2)w_1^2\right)^s-1\right]^{p-1}}{\left(1+y^2\right)^{\frac{n+sp}{2}}}y^{n-2}\mathrm{d}y\mathrm{d}w_1\\
&\quad+(1-x)^{-s}\int_{\frac{3}{4}}^{\frac{2x^2}{1-x^2}}\frac{1}{w_1^{1+sp}}\int_{\sqrt{\frac{2x^2}{1-x^2}\frac{1}{w_1}-1} }^{\sqrt{\frac{4x^2}{1-x^2}\frac{1}{w_1}-1}}\frac{\left[\left(1+w_1\right)^s-1\right]^{p-1}-\left[\left(1+w_1-\frac{1-x^2}{4x^2}(1+y^2)w_1^2\right)^s-1\right]^{p-1}}{\left(1+y^2\right)^{\frac{n+sp}{2}}}y^{n-2}\mathrm{d}y\mathrm{d}w_1\\
%&\leq C(1-x)^{-s}\int_{\frac{3}{4}}^{\frac{2x^2}{1-x^2}}\frac{1}{w_1^{1+sp}}\int_{0 }^{\sqrt{\frac{2x^2}{1-x^2}\frac{1}{w_1}-1}}\frac{w_1^{s(p-2)}\left[\left(1+w_1\right)^s-\left(1+w_1-\frac{1-x^2}{4x^2}(1+y^2)w_1^2\right)^s\right]}{\left(1+y^2\right)^{\frac{n+sp}{2}}}y^{n-2}\mathrm{d}y\mathrm{d}w_1\\
%&\quad+C(1-x)^{-s}\int_{\frac{3}{4}}^{\frac{2x^2}{1-x^2}}\frac{1}{w_1^{1+sp}}\int_{0 }^{\sqrt{\frac{2x^2}{1-x^2}\frac{1}{w_1}-1}}\frac{w_1^{s(p-2)}\left[\left(1+w_1\right)^s-\left(1+w_1-\frac{1-x^2}{4x^2}(1+y^2)w_1^2\right)^s\right]}{\left(1+y^2\right)^{\frac{n+sp}{2}}}y^{n-2}\mathrm{d}y\mathrm{d}w_1\\
&\leq C(1-x)^{-s}\int_{\frac{3}{4}}^{\frac{2x^2}{1-x^2}}\frac{1}{w_1^{1+sp}}\int_{0 }^{\sqrt{\frac{2x^2}{1-x^2}\frac{1}{w_1}-1}}\frac{w_1^{s(p-2)}\left(1+w_1\right)^s\left[1-\left(1-\frac{1-x^2}{4x^2}(1+y^2)\frac{w_1^2}{1+w_1}\right)^s\right]}{\left(1+y^2\right)^{\frac{n+sp}{2}}}y^{n-2}\mathrm{d}y\mathrm{d}w_1\\
&\quad+C(1-x)^{-s}\int_{\frac{3}{4}}^{\frac{2x^2}{1-x^2}}\frac{1}{w_1^{1+sp}}\frac{1}{\left(\frac{2x^2}{1-x^2}\frac{1}{w_1}\right)^{\frac{sp}{2}}}\mathrm{d}w_1\\
&\leq C(1\!-\!x)^{-s}\int_{\frac{3}{4}}^{\frac{2x^2}{1-\!x^2}}\frac{1}{w_1^{1+sp}}\int_{0 }^{\sqrt{\frac{2x^2}{1-x^2}\frac{1}{w_1}\!-\!1}}\frac{w_1^{s(p-2)}\left(1\!+\!w_1\right)^s\frac{1-\!x^2}{4x^2}(1+\!y^2)\frac{w_1^2}{1+w_1}}{\left(1+y^2\right)^{\frac{n+sp}{2}}}y^{n-2}\mathrm{d}y\mathrm{d}w_1
\\
&\quad+C(1-x)^{\frac{s(p-2)}{2}}\int_{\frac{3}{4}}^{\frac{2x^2}{1-x^2}}\frac{1}{w_1^{1+\frac{sp}{2}}}\mathrm{d}w_1\\
&\leq C(1-x)^{1-s}\int_{\frac{3}{4}}^{\frac{2x^2}{1-x^2}}\frac{w_1^{1-2s}}{(1+w_1)^{1-s}}\int_{0 }^{\sqrt{\frac{2x^2}{1-x^2}\frac{1}{w_1}-1}}\frac{1}{\left(1+y^2\right)^{\frac{sp}{2}}}\mathrm{d}y\mathrm{d}w_1+C(1-x)^{\frac{s(p-2)}{2}}.\\
\end{aligned}
\end{equation*}
Besides,
\begin{equation*}
\begin{aligned}
&\quad (1-x)^{1-s}\int_{\frac{3}{4}}^{\frac{2x^2}{1-x^2}}\frac{w_1^{1-2s}}{(1+w_1)^{1-s}}\int_{0 }^{\sqrt{\frac{2x^2}{1-x^2}\frac{1}{w_1}-1}}\frac{1}{\left(1+y^2\right)^{\frac{sp}{2}}}\mathrm{d}y\mathrm{d}w_1\\
&\leq C(1-x)^{1-s}\int_{\frac{3}{4}}^{\frac{2x^2}{1-x^2}}w_1^{-s}\int_{0 }^{\sqrt{\frac{2x^2}{1-x^2}\frac{1}{w_1}-1}}\frac{1}{\left(1+y^2\right)^{\frac{sp}{2}}}\mathrm{d}y\mathrm{d}w_1\\
&\leq\begin{cases}
(1-x)^{1-s}\int_{\frac{3}{4}}^{\frac{2x^2}{1-x^2}}w_1^{-s}\mathrm{d}w_1\leq C,&\text{if}\quad sp\geq 2;\\
\\
C(1-x)^{1-s}\int_{\frac{3}{4}}^{\frac{2x^2}{1-x^2}}w_1^{-s}\left(\frac{2x^2}{1-x^2}\frac{1}{w_1}\right)^{1-\frac{sp}{2}}\mathrm{d}w_1\leq C(1-x)^{\frac{s(p-2)}{2}}\int_{\frac{3}{4}}^{\frac{2x^2}{1-x^2}}w_1^{\frac{s(p-2)}{2}-1}\mathrm{d}w_1\leq C,&\text{if}\quad sp< 2.
\end{cases}
\end{aligned}
\end{equation*}
Secondly, we claim that  $\left( \uppercase\expandafter{\romannumeral4}\right) $ can be substituted by 
\begin{equation*}
\int_{\frac{1}{2}}^{\frac{4x^2}{1-x^2}}\frac{1}{w_1^{1+sp}}\int_{0 }^{\infty}\frac{\left[\left(1+w_1\right)^s-1\right]^{p-1}}{\left(1+y^2\right)^{\frac{n+sp}{2}}}y^{n-2}\mathrm{d}y\mathrm{d}w_1.
\end{equation*}
Since 
\begin{equation*}
\begin{aligned}
&\quad (1-x)^{-s}\int_{\frac{1}{2}}^{\frac{4x^2}{1-x^2}}\frac{1}{w_1^{1+sp}}\int_{\sqrt{\frac{4x^2}{1-x^2}\frac{1}{w_1}-1} }^{\infty}\frac{\left[\left(1+w_1\right)^s-1\right]^{p-1}}{\left(1+y^2\right)^{\frac{n+sp}{2}}}y^{n-2}\mathrm{d}y\mathrm{d}w_1\\
&\leq C (1-x)^{-s}\int_{\frac{1}{2}}^{\frac{4x^2}{1-x^2}}\frac{\left[\left(1+w_1\right)^s-1\right]^{p-1}}{w_1^{1+sp}}\frac{1}{\left(\frac{4x^2}{1-x^2}\frac{1}{w_1}\right)^{\frac{sp}{2}}}\mathrm{d}y\mathrm{d}w_1\\
&\leq C(1-x)^{\frac{s(p-2)}{2}}\int_{\frac{1}{2}}^{\frac{4x^2}{1-x^2}}\frac{\left[\left(1+w_1\right)^s-1\right]^{p-1}}{w_1^{1+\frac{sp}{2}}}\mathrm{d}w_1\\
&\leq C(1-x)^{\frac{s(p-2)}{2}}\int_{\frac{1}{2}}^{\frac{4x^2}{1-x^2}}\frac{w_1^{s(p-1)}}{w_1^{1+\frac{sp}{2}}}\mathrm{d}w_1\\
&\leq C(1-x)^{\frac{s(p-2)}{2}}\int_{\frac{1}{2}}^{\frac{4x^2}{1-x^2}}w_1^{\frac{s(p-2)}{2}-1}\mathrm{d}w_1\leq C.\\
\end{aligned}
\end{equation*}
Thirdly,  $\left( \uppercase\expandafter{\romannumeral4}\right) $ can be replaced by 
\begin{equation*}
\int_{\frac{1}{2}}^{\infty}\frac{1}{w_1^{1+sp}}\int_{0 }^{\infty}\frac{\left[\left(1+w_1\right)^s-1\right]^{p-1}}{\left(1+y^2\right)^{\frac{n+sp}{2}}}y^{n-2}\mathrm{d}y\mathrm{d}w_1.
\end{equation*}
Due to the fact that 
\begin{equation*}
\begin{aligned}
&\quad(1-x)^{-s}\int_{\frac{4x^2}{1-x^2}}^{\infty}\frac{1}{w_1^{1+sp}}\int_{0 }^{\infty}\frac{\left[\left(1+w_1\right)^s-1\right]^{p-1}}{\left(1+y^2\right)^{\frac{n+sp}{2}}}y^{n-2}\mathrm{d}y\mathrm{d}w_1\leq C(1-x)^{-s}\int_{\frac{4x^2}{1-x^2}}^{\infty}\frac{w_1^{s(p-1)}}{w_1^{1+sp}}\mathrm{d}w_1\leq C.
\end{aligned}
\end{equation*}

Step 14. Overall,  the singular term is 
\begin{equation*}
(1-x)^{-s}\int_{0}^{\infty}\frac{y^{n-2}}{(1+y^2)^\frac{n+sp}{2}}\Bigg\{\frac{1}{sp}+\int_{0}^{1}\frac{\left[1-(1-w_1)^s\right]^{p-1}-\left[(1+w_1)^s-1\right]^{p-1}}{w_1^{1+sp}}\mathrm{d}w_1-\int_{1}^{\infty}\frac{\left[(1+w_1)^s-1\right]^{p-1}}{w_1^{1+sp}}\mathrm{d}w_1\Bigg\}.
\end{equation*}
Hence we only need to prove the following identity,
\begin{equation*}
\frac{1}{sp}+\int_{0}^{1}\frac{\left[1-(1-w_1)^s\right]^{p-1}-\left[(1+w_1)^s-1\right]^{p-1}}{w_1^{1+sp}}\mathrm{d}w_1-\int_{1}^{\infty}\frac{\left[(1+w_1)^s-1\right]^{p-1}}{w_1^{1+sp}}\mathrm{d}w_1=0.
\end{equation*}
Which is exactly the identity (\ref{identity}). 
%Step 15. For all $x\in(0,1-\delta)$ with $\delta>0.$ By \cite[Lemma 5.2]{Chen2018Maximum}, 
%\begin{equation}
%\lvert(-\Delta)^s_p u(x)\rvert\leq C\int\limits_{\mathbb{R}^n\backslash B_\delta(x)}\frac{1}{|x-y|^{n+sp}}\mathrm{d}y+C\lvert\nabla^2 u(x)\rvert^{p-2}\int_{B_\delta(x)}\frac{|x-y|^p}{|x-y|^{n+sp}}\mathrm{d}y\leq C.\\
%\end{equation}
Hence we have completed the proof for higher dimensions.
\section{Hopf's lemma}
In order to elaborate on the difference between $C^{1,1}_{loc}(\Omega)\cap \mathcal{L}_{sp}$ and $\mathcal{W}^{s,p}(\Omega)$, we give a function in $\big\{C^{1,1}_{loc}(\Omega)\cap \mathcal{L}_{sp}\big\}\setminus\mathcal{W}^{s,p}(\Omega)$ for a bounded domain $\Omega$. %where$$\mathcal{W}^{s,p}(\Omega):=\{u\in L^p_{loc}(\mathbb{R}^n)\big\rvert \exists U\supset\supset \Omega,\quad such~that\\~ \iint\limits_{U\times U}\frac{|u(x)-u(y)|^p}{|x-y|^{n+sp}}\mathrm{d}x\mathrm{d}y<+\infty\}.$$
\begin{exmp}\label{lizi}
	Set 
	\begin{equation}
	u(x)=\begin{cases}
	\frac{1}{|x|^t},& \text{if} \quad x\in B^+_1(0):=\{x\in B_1(0)\rvert x_n>0\};\\
	0, & \text{if}\quad x\in \mathbb{R}^n\setminus B^+_1(0).
	\end{cases}
	\end{equation}
	Then $u\in C^{1,1}_{loc}\left(B^+_1(0)\right)$ and $u$ is lower semicontinuous,
	\begin{equation}
	\begin{aligned}
	\int_{\mathbb{R}^n} \frac{|1+u(x)|^{p-1}}{1+|x|^{n+sp}}\mathrm{d}x&=\int_{B^+_1(0)} \frac{|1+\frac{1}{|x|^t}|^{p-1}}{1+|x|^{n+sp}}\mathrm{d}x+\int_{\mathbb{R}^n\setminus B^+_1(0)} \frac{1}{1+|x|^{n+sp}}\mathrm{d}x\\
	&\leq C(p) \int_{B^+_1(0)} \frac{\frac{1}{|x|^{t(p-1)}}}{1+|x|^{n+sp}}\mathrm{d}x+C(p)\int_{\mathbb{R}^n} \frac{1}{1+|x|^{n+sp}}\mathrm{d}x\\
	&\leq C(p)\int_{B^+_1(0)} \frac{1}{|x|^{t(p-1)}}\mathrm{d}x+C(p)\\
	&\leq C(p,n)\int_{0}^{1}\rho^{n-t(p-1)-1}\mathrm{d}\rho+C(p,n)\\
	&< +\infty \quad\text{if} \quad t<\frac{n}{p-1}.
	\end{aligned}
	\end{equation}
	So $u(x)\in C^{1,1}_{loc}\left(B^+_1(0)\right)\cap \mathcal{L}_{sp}\left(\mathbb{R}^n\right)$ when $t<\frac{n}{p-1}.$
	However, when $t\geq\frac{n}{p},$ $u(x)\notin L^p(B^+_1(0))$.
	%	However when considering the $\mathcal{W}^{s,p}$ seminorm,
	%\begin{equation}
	%\begin{aligned}
	%\iint\limits_{B^+_1(0)\times B^+_1(0)}\frac{|u(x)-u(y)|^p}{|x-y|^{2+sp}}\mathrm{d}x\mathrm{d}y&=\iint\limits_{B^+_1(0)\times B^+_1(0)}\frac{|\frac{1}{|x|^t}-\frac{1}{|y|^t}|^p}{|x-y|^{n+sp}}\mathrm{d}x\mathrm{d}y\\
	%&\geq\int\limits_{B^+_1(0)}\left(\int\limits_{B^+_1(0)\cap\{x: |x|<\frac{|y|}{2}\}}\frac{|\frac{1}{|x|^t}-\frac{1}{|y|^t}|^p}{|x-y|^{n+sp}}\mathrm{d}x\right)\mathrm{d}y\\
	%	&=\int\limits_{B^+_1(0)}\left(\int\limits_{B^+_1(0)\cap\{x: |x|<\frac{|y|}{2}\}\}}\frac{|\frac{|y|^t-|x|^t}{|x|^t|y|^t}   |^p}{|x-y|^{n+sp}}\mathrm{d}x\right)\mathrm{d}y\\
	%	&\geq C\int\limits_{B^+_1(0)}\frac{1}{|y|^{2tp}}\left(\int\limits_{B^+_1(0)\cap\{x: |x|<\frac{|y|}{2}\}} \frac{(|y|^t-|x|^t  )^p}{|x-y|^{n+sp}}\mathrm{d}x\right)\mathrm{d}y\\
	%	&\geq C\int\limits_{B^+_1(0)}\frac{1}{|y|^{2tp}}\left(\int\limits_{B^+_1(0)\cap\{x: |x|<\frac{|y|}{2}\}} \frac{|y|^{tp}}{|y|^{n+sp}}\mathrm{d}x\right)\mathrm{d}y\\
	%	&\geq C\int\limits_{B^+_1(0)}\frac{1}{|y|^{(t+s)p}}\mathrm{d}y\\
	%	&=C \int_{0}^{1}\rho^{n-(t+s)p-1}\mathrm{d}\rho\\
	%	&=+\infty \quad \text{if} \quad t>\frac{n}{p}-s.
	%	\end{aligned}
	%	\end{equation}
	So for $p\in(\frac{n}{p},\frac{n}{p-1})$,  $u(x)\in\big\{C^{1,1}_{loc}\left(B^+_1(0)\right)\cap \mathcal{L}_{sp}\big\}\setminus\mathcal{W}^{s,p}\left(B^+_1(0)\right).$ 
	%Hence $C^{1,1}_{loc}(\Omega)\cap \mathcal{L}_{sp}$ is  different from the space $\mathcal{W}^{s,p}(\Omega).$
\end{exmp}
From the example above, it is meaningful to investigate the Hopf's lemma for $u(x)\in C^{1,1}_{loc}(\Omega)\cap \mathcal{L}_{sp}$  in the point-wise sense.
%\begin{Theorem}
%	 For any $u(x)\in C^{1,1}_{loc}(\Omega)\cap \mathcal{L}_{sp}(\mathbb{R}^n)$ and $u$ is lower semicontinuous on $\overline{\Omega}$ point-wisely satisfies
%	\begin{equation}
%	\begin{cases}
%	(-\Delta)^s_p u\geq 0, & x\in \Omega,\\
%	u>0, & x\in \Omega,\\
%	u=0, & x\in \mathbb{R}^n\setminus\Omega.
%	\end{cases}
%	\end{equation}
%	Where $\Omega$ is any bounded domain with a uniform interior ball condition (e.g. a domain of class $C^{1,1}$). Then there is a constant $C=C(\Omega, u)$, such that \[\liminf_{x\rightarrow \partial\Omega}\frac{u(x)}{\left[dist(x,\partial\Omega)\right]^s}\geq C.\] 	   
%\end{Theorem}
\begin{proof}[Proof of Theorem \ref{thm 2}]
	%[Proof of Theorem \ref{}]
	Since u satisfies the uniform interior ball condition, we assume the uniform radius is $10\varrho.$ Then for every $z\in\partial\Omega$, there is a ball $B_{\varrho}(y)\subset\Omega$ centered at $y\in\Omega$ with $\{z\}=\partial B_{\varrho}(y)\cap\partial\Omega.$ And $\delta(x):=dist(x,\partial\Omega)=|x-z|$ for all $x\in[y,z].$ 
	Without loss of generality, we relocate the origin to $y$ so that $\delta(x)=\varrho-|x|.$ 
	
	Set $$\psi(x)=(1-|x|^2)^s_{+}, \quad \psi_{\varrho}(x)=\psi(\frac{x}{\varrho})=\frac{(\varrho^2-|x|)^s_{+}}{\varrho^{2s}},\quad x\in B_{\varrho}.$$ Then by Theorem \ref{thm 1}, there exists a constant $C_0$>0, such that  $$(-\Delta)^s_p \psi(x)\leq C_0,\quad(-\Delta)^s_p \psi_{\varrho}(x)=\frac{(-\Delta)^s_p \psi(\frac{x}{\varrho})}{\varrho^{sp}}\leq \frac{C_0}{\varrho^{sp}}.$$
	
	Since $u>0$ in $\Omega$, we consider a region $D:=\{x\in\Omega: dist(x,\partial\Omega)\geq 3\varrho\}$ which has a positive distance with $B_{\varrho}.$  So $\displaystyle C_{D}:=\inf\limits_D u(x) >0.$ 
	Then we set $\displaystyle u_{-}(x)=u\chi_{D}+\epsilon\psi_{\varrho},$ where $\epsilon<C_{D} $ is to be specified later. Denote $[t]^{p-1}:=|t|^{p-2}t.$ For any $x\in B_{\varrho} ,$
	\begin{equation*}
	\begin{aligned}
	(-\Delta)^s_p u_{-}(x)&=C_{n,s,p}P.V.\int_{\mathbb{R}^n}\frac{\left[\epsilon\psi_{\varrho}(x)-u_{-}(y)\right]^{p-1}}{|x-y|^{n+sp}}\mathrm{d}y\\
	&=C_{n,s,p}P.V.\Big\{\int_{B_{\varrho}(y)}\frac{\left[\epsilon\psi_{\varrho}(x)-\epsilon\psi_{\varrho}(y)\right]^{p-1}}{|x-y|^{n+sp}}\mathrm{d}y+\int_{D}\frac{\left[\epsilon\psi_{\varrho}(x)-u(y)\right]^{p-1}}{|x-y|^{n+sp}}\mathrm{d}y\\
	&\quad\quad\quad\quad\quad+\int_{\mathbb{R}^n\setminus\left(D\cup B_{\varrho}(y)\right)}\frac{\left[\epsilon\psi_{\varrho}(x)\right]^{p-1}}{|x-y|^{n+sp}}\mathrm{d}y\Big\}\\
	&=\epsilon^{p-1}(-\Delta)^s_p \psi_{\varrho}(x)+C_{n,s,p}P.V.\Big\{\int_{D}\frac{\left[\epsilon\psi_{\varrho}(x)-u(y)\right]^{p-1}}{|x-y|^{n+sp}}\mathrm{d}y-\int_{D}\frac{\left[\epsilon\psi_{\varrho}(x)\right]^{p-1}}{|x-y|^{n+sp}}\mathrm{d}y\Big\}\\
	&\leq\frac{C_0\epsilon^{p-1}}{\varrho^{sp}}+C_{n,s,p}P.V.\Big\{\int_{D}\frac{\left[\epsilon\psi_{\varrho}(x)-u(y)\right]^{p-1}-\left[\epsilon\psi_{\varrho}(x)\right]^{p-1}}{|x-y|^{n+sp}}\mathrm{d}y\Big\}\\
	&\leq \frac{C_0\epsilon^{p-1}}{\varrho^{sp}}+C_{n,s,p}2^{2-p}P.V.\Big\{\int_{D}\frac{\left[-u(y)\right]^{p-1}}{|x-y|^{n+sp}}\mathrm{d}y\Big\}\\
	&\leq \frac{C_0\epsilon^{p-1}}{\varrho^{sp}}-C_{n,s,p}2^{2-p}C_D^{p-1}\int_D\frac{1}{|x-y|^{n+sp}}\mathrm{d}y\\
	&\leq \frac{C_0\epsilon^{p-1}}{\varrho^{sp}}-C_{n,s,p}2^{2-p}C_D^{p-1}C(\varrho),
	\end{aligned}
	\end{equation*} 
	where we use the inequality in Lemma \ref{dandiao}.
	So we can choose $\epsilon$ small such that $(-\Delta)^s_p u_{-}(x)+c(x)u_{-}\leq 0$ for any $x\in B_{\varrho} . $ Then by  Propositon \ref{bijiaoyuanli}, $$u(x)\geq u_{-}(x)=\epsilon\psi_{\varrho}(x)=\frac{\epsilon(\varrho^2-|x|)^s}{\varrho^{2s}}=\frac{\epsilon(\varrho+|x|)^s(\varrho-|x|)^s}{\varrho^{2s}}=\frac{\epsilon(\varrho+|x|)^s}{\varrho^{2s}}\delta(x)^s,\quad\forall x\in [y,z].$$
\end{proof}
\section{Global H\"older regularity of bounded positive solutions} 
\begin{proof}[Proof of Theorem \ref{thm3}]
	First we repeat the last part of the proof in \cite[Theorem 2]{jin2019}, there exist $\epsilon_0, \nu_0$ such that $$|u(x)|\leq c\left[dist(x,\partial\Omega)\right]^{\nu_0},\quad\forall x\in V:=\{x\in\Omega\ |\  dist\{x,\partial\Omega\}<\epsilon_0\}.$$
	By the boundness of $u$, we have \[|u(x)|\leq c\left[dist(x,\partial\Omega)\right]^{\nu_0},\quad\forall x\in\Omega.\]
	
	Now we consider $x\in W:=\{x\in\Omega\  |\  dist\{x,\partial\Omega\}\geq\epsilon_0\},$ then $\displaystyle B_{\frac{\epsilon_0}{2}}(x)\subset\subset \Omega, \ \forall x\in W.$	
	We claim  $u\in\mathcal{W}^{s,p}_{loc}(\Omega)$ if $u\in C^{1,1}_{loc}(\Omega)\cap \mathcal{L}_{sp}.$  
	In fact, for any domain $\Omega'\subset\subset\Omega,$
	\begin{equation*}
	\begin{aligned}
	\iint\limits_{\Omega'\times\Omega'}\frac{|u(x)-u(y)|^p}{|x-y|^{n+sp}}\mathrm{d}x\mathrm{d}y&=\int\limits_{\Omega'}\left(\int\limits_{\Omega'}\frac{\lvert\nabla u(y)\cdot(x-y)+o(|x-y|)\rvert^p}{|x-y|^{n+sp}}\mathrm{d}x\right)\mathrm{d}y\\
	&\leq C(\Omega') \int\limits_{\Omega'}\left(\int\limits_{\Omega'}\frac{\lvert x-y\rvert^p}{|x-y|^{n+sp}}\mathrm{d}x\right)\mathrm{d}y\leq C(\Omega').
	\end{aligned}
	\end{equation*}
	By \cite[Proposition 2.12, Lemma 2.5]{iannizzotto2014global}, $u$ is a weak solution of $(-\Delta)^s_p u=f$ in $ \displaystyle B_{\frac{\epsilon_0}{2}}(x),\ \forall x\in W.$ Then by \cite[Theorem 1.4]{brasco2018higher}, for this $\nu_0$, we have 
	\begin{equation*}
	\begin{aligned}
	\left[u\right]_{C^{\nu_0}\left(B_{\frac{\epsilon_0}{64}}(x)\right)}&\leq\frac{C({\nu_0})}{{\epsilon_0}^{\nu_0}}\left[\|u\|_{L^\infty\left(B_{\frac{\epsilon_0}{8}}(x)\right)}+\left(\epsilon_0^{sp}\int_{\mathbb{R}^n\setminus\left(B_{\frac{\epsilon_0}{8}}(x)\right)}\frac{|u|^{p-1}}{|x-y|^{n+sp}}\mathrm{d}y\right)^{\frac{1}{p-1}}+\left(\epsilon_0^{sp}\|f\|_{L^\infty\left(B_{\frac{\epsilon_0}{8}}(x)\right)}\right)^{\frac{1}{p-1}}\right]\\
	&\leq \frac{C({\nu_0})}{{\epsilon_0}^{\nu_0}}\left[\|u\|_{L^\infty\left(\Omega\right)}+\epsilon_0^{\frac{sp}{p-1}}\|f\|^{\frac{1}{p-1}}_{L^\infty(\Omega)}\right]\leq C(\nu_0,\epsilon_0)\left[\|u\|_{L^\infty\left(\Omega\right)}+\|f\|^{\frac{1}{p-1}}_{L^\infty(\Omega)}\right].
	\end{aligned}
	\end{equation*}
	So for $\forall x, y \in W,$ we have $$\frac{|u(x)-u(y)|}{|x-y|^{\nu_0}}\leq C(\nu_0,\epsilon_0)\left[\|u\|_{L^\infty\left(\Omega\right)}+\|f\|^{\frac{1}{p-1}}_{L^\infty(\Omega)}\right].$$
	
	The next part is similar to the \cite[Section 5.2]{iannizzotto2014global}.
	Now we consider that $x,y\in V.$ Again by \cite[Theorem 1.4]{brasco2018higher}, we denote $\delta_x:= dist(x,\partial\Omega),$ 
	\begin{equation}\label{holder}
	\begin{aligned}
	\left[u\right]_{C^{\nu_0}\left(B_{\frac{\delta_x}{64}}(x)\right)}&\leq\frac{C({\nu_0})}{{\delta_x}^{\nu_0}}\left[\|u\|_{L^\infty\left(B_{\frac{\delta_x}{8}}(x)\right)}+\left(\delta_x^{sp}\int_{\mathbb{R}^n\setminus\left(B_{\frac{\delta_x}{8}}(x)\right)}\frac{|u|^{p-1}}{|x-y|^{n+sp}}\mathrm{d}y\right)^{\frac{1}{p-1}}+\left(\delta_x^{sp}\|f\|_{L^\infty\left(B_{\frac{\delta_x}{8}}(x)\right)}\right)^{\frac{1}{p-1}}\right]\\
	&\leq \frac{C({\nu_0})}{{\delta_x}^{\nu_0}}\left[C\delta_x^{\nu_0}+\left(\delta_x^{sp}\int_{\Omega\setminus\left(B_{\frac{\delta_x}{8}}(x)\right)}\frac{C\delta_y^{\nu_0(p-1)}}{|x-y|^{n+sp}}\mathrm{d}y\right)^{\frac{1}{p-1}}+\left(\delta_x^{sp}\|f\|_{L^\infty\left(B_{\frac{\delta_x}{8}}(x)\right)}\right)^{\frac{1}{p-1}}\right]\\
	&\leq C({\nu_0})\left(1+\|f\|^{\frac{1}{p-1}}_{L^\infty(\Omega)}\right)+\frac{C({\nu_0})}{{\delta_x}^{\nu_0}}\left(\delta_x^{sp}\int_{\Omega\setminus\left(B_{\frac{\delta_x}{8}}(x)\right)}\frac{C\delta_y^{\nu_0(p-1)}}{|x-y|^{n+sp}}\mathrm{d}y\right)^{\frac{1}{p-1}}\\
	&\leq C({\nu_0})\left(1+\|f\|^{\frac{1}{p-1}}_{L^\infty(\Omega)}\right)+\frac{C({\nu_0})}{{\delta_x}^{\nu_0}}\left(\delta_x^{sp}\int_{\Omega\setminus\left(B_{\frac{\delta_x}{8}}(x)\right)}\frac{C(\delta_x+|x-y|)^{\nu_0(p-1)}}{|x-y|^{n+sp}}\mathrm{d}y\right)^{\frac{1}{p-1}}\\
	&\leq C({\nu_0})\left(1+\|f\|^{\frac{1}{p-1}}_{L^\infty(\Omega)}\right)+\frac{C({\nu_0})}{{\delta_x}^{\nu_0}}\left(\delta_x^{sp}\int_{\Omega\setminus\left(B_{\frac{\delta_x}{8}}(x)\right)}\frac{C|x-y|^{\nu_0(p-1)}}{|x-y|^{n+sp}}\mathrm{d}y\right)^{\frac{1}{p-1}}\\
	&\leq C({\nu_0})\left(1+\|f\|^{\frac{1}{p-1}}_{L^\infty(\Omega)}\right)+\frac{C({\nu_0})}{{\delta_x}^{\nu_0}}\left(\delta_x^{sp}\int_{\mathbb{R}^n\setminus\left(B_{\frac{\delta_x}{8}}(x)\right)}\frac{C|x-y|^{\nu_0(p-1)}}{|x-y|^{n+sp}}\mathrm{d}y\right)^{\frac{1}{p-1}}\\
	&\leq C({\nu_0})\left(1+\|f\|^{\frac{1}{p-1}}_{L^\infty(\Omega)}\right).
	\end{aligned}
	\end{equation}
	Now for $x,y\in V$ and without loss of generality, we assume $\delta_x\geq \delta_y.$
	\begin{itemize}
		\item When $\displaystyle |x-y|<\frac{\delta_x}{64},$ i.e. $\displaystyle y\in B_{\frac{\delta_x}{64}}(x)$, by \eqref{holder}, we have
		$$ \frac{|u(x)-u(y)|}{|x-y|^{\nu_0}}\leq C({\nu_0})\left(1+\|f\|^{\frac{1}{p-1}}_{L^\infty(\Omega)}\right).$$
		\item When $\displaystyle |x-y|\geq\frac{\delta_x}{64}\geq \frac{\delta_y}{64},$ we have $$ |u(x)-u(y)|\leq|u(x)|+|u(y)|\leq \displaystyle C(\delta_x^{\nu_0}+\delta_y^{\nu_0})\leq C|x-y|^{\nu_0}.$$
	\end{itemize}
	Hence our conclusion is $\displaystyle u\in C^{\nu_0}(\mathbb{R}^n)$
\end{proof}
\section{Appendix}
We list the numerical calculation results below. We refer the readers who are interested in the details to  \cite{zaizheng2020}, which is the Ph.D. thesis of the first author.
\begin{Proposition}[Numerical calculation]\label{fulu}
	Let $n=1,$ $s=\frac{1}{2},$ $ p=2, 4, 6, 8,$  $u(x)=(1-x^2)^s_{+}$. Then $(-\Delta)^s_p u(x)$ is uniformly bounded in $(-1,1)$. More precisely (omit constant $C_{n,s,p}$), we have
	\begin{equation*}
	\begin{aligned}
	(-\Delta)^{\frac{1}{2}} u(x)&\equiv2\arcsin 1\equiv\pi;\\
	(-\Delta)^{\frac{1}{2}}_4 u(x)&=3\sqrt{1-x^2}\left[\log (4-4x^2)-1\right]+6x\arcsin x;\\
	(-\Delta)^{\frac{1}{2}}_6 u(x)&=20\sqrt{1-x^2}\left[x\log\frac{1-x}{1+x}+2\right]+\frac{5\pi}{2}(8x^2-5);\\
	(-\Delta)^{\frac{1}{2}}_8 u(x)&=7\sqrt{1-x^2}\left[4(5x^2-2)\log(4-4x^2)\!+\!\frac{(-x^2+67)}{6}\right]\!+\!35x(8x^2-7)\arcsin x.
	\end{aligned}
	\end{equation*}
	In addition, $\displaystyle(-\Delta)^{\frac{1}{2}}_4 u(x)$, $\displaystyle(-\Delta)^{\frac{1}{2}}_6 u(x)$, $\displaystyle(-\Delta)^{\frac{1}{2}}_8 u(x)$  are strictly increasing in $(0,1)$ and 
	\bigskip $\displaystyle(-\Delta)^{\frac{1}{2}}_4 u(x)\in\left[6\log 2-3,3\pi\right)$,
	$\displaystyle(-\Delta)^{\frac{1}{2}}_6 u(x)\in\left[40-\frac{25\pi}{2},\frac{15\pi}{2}\right)$, $\displaystyle(-\Delta)^{\frac{1}{2}}_8 u(x)\in\left[\frac{469}{6}-112\log2, \frac{35\pi}{2}\right)$.
	% Actually, for all even parameters, we can prove the operator is bounded.
\end{Proposition}
\section*{Conflict of interest statement}
The authors report no conflicts of interest. The authors are responsible for the content and writing of this article.
\section*{Acknowledgement}
The first author was supported by the CHINA SCHOLARSHIP COUNCIL. 
\bibliographystyle{plain}
\bibliography{wenxian}
\nocite{*}

%\begin{thebibliography}{AAA}

%\bibitem{hello} hello

%\end{thebibliography}

\end{document}